\documentclass[a4paper,reqno]{amsart}

\pdfoutput=1

\usepackage{microtype}
\usepackage[utf8]{inputenc}

\usepackage{amsmath,amstext,amssymb,mathrsfs,amscd,amsthm,indentfirst,bbm}
\usepackage{amsfonts}
\usepackage{stmaryrd}

\usepackage{todonotes}

\usepackage[style=ieee-alphabetic, url=false, backref=true, doi=false, isbn=false, eprint=false]{biblatex}
\DeclareFieldFormat{url}{\textsc{url}\addcolon\space\url{#1}}

\usepackage{float}
\usepackage{adjustbox}
\usepackage{booktabs}
\usepackage{verbatim}
\usepackage{enumerate}
\usepackage{graphicx}
\usepackage{textcomp}
\usepackage{tikz,tikz-cd}
\usetikzlibrary{decorations.pathreplacing}
\usetikzlibrary{patterns}
\usepackage{ifthen}
\usepackage{url}

\usepackage{xcolor}
\definecolor{lightgreen}{RGB}{45,137,105}
\definecolor{purple}{RGB}{155,26,103}
\usepackage{hyperref}
\hypersetup{bookmarksdepth=3}
\PassOptionsToPackage{hyphens}{url}\usepackage{hyperref}
\hypersetup{
    colorlinks  = true,
    linkcolor   = purple,
    citecolor   = lightgreen,
    urlcolor    = [rgb]{0,.40,.80}
}
\usepackage[nameinlink,capitalize]{cleveref}
\usepackage{mathtools}
\usepackage{mleftright}
\usepackage{enumitem}

\usepackage[new]{old-arrows}

\usepackage{libertine}
\makeatletter
\renewcommand\th@plain{\slshape}
\makeatother

\calclayout

\newcommand{\bC}{\mathbb C}

\newcommand{\bP}{\mathbb P}
\newcommand{\bQ}{\mathbb Q}
\newcommand{\bR}{\mathbb R}
\newcommand{\bS}{\mathbb S}

\newcommand{\bZ}{\mathbb Z}

\newcommand{\cC}{\mathcal C}

\newcommand{\cE}{\mathcal E}
\newcommand{\cF}{\mathcal F}

\newcommand{\cH}{\mathcal H}

\newcommand{\cL}{\mathcal L}
\newcommand{\cM}{\mathcal M}

\newcommand{\cO}{\mathcal O}
\newcommand{\cP}{\mathcal P}

\newcommand{\cT}{\mathcal T}

\newcommand{\cY}{\mathcal Y}
\newcommand{\cZ}{\mathcal Z}

\newcommand{\fm}{\mathfrak m}

\newcommand{\lra}{\longrightarrow}

\newcommand{\Gammahol}{\Gamma_{\mathrm{hol}}}
\newcommand{\Gammacon}{\Gamma_{\cC^0}}

\newcommand{\Hol}{\mathrm{Hol}}
\newcommand{\Map}{\mathrm{Map}}

\newcommand{\UConf}{\mathrm{UConf}}

\newcommand{\Diff}{\mathrm{Diff}}
\newcommand{\Conf}{\mathrm{Conf}}

\DeclareMathOperator{\Sym}{Sym}

\newcommand{\catSp}{\mathsf{Sp}}

\newcommand{\catS}{\mathsf{S}}

\newcommand{\opp}{\mathrm{op}}

\newcommand{\catTop}{\mathsf{Top}}

\DeclareMathOperator{\fib}{fib}
\DeclareMathOperator{\cofib}{cofib}

\DeclareMathOperator*{\colim}{colim}

\newcommand{\ovwedge}{\overline{\wedge}}
\newcommand{\sh}{\mathrm{sh}}

\DeclareMathOperator{\tub}{tub}
\newcommand{\can}{\mathrm{can}}

\newcommand{\MTSO}{\mathrm{MTSO}}
\DeclareMathOperator{\conn}{conn}

\NewDocumentCommand{\set}{somm}{%
   \IfNoValueTF{#2}
    {\IfBooleanTF{#1}{\{#3 \mid #4\}}{\mleft\{ #3 \mathrel{}\middle\vert\mathrel{} #4 \mright\}}}
    {\mathopen{#2\{}#3 \mathrel{}#2\vert\mathrel{} #4\mathclose{#2\}}}%
  }

\newtheorem{theorem}{Theorem}[section]
\newtheorem*{theorem*}{Theorem}
\newtheorem{theorem-in-progress}[theorem]{Theorem-in-progress}

\newtheorem*{conjecture*}{Conjecture}

\newtheorem{definition}[theorem]{Definition}
\newtheorem{lemma}[theorem]{Lemma}
\newtheorem{remark}[theorem]{Remark}
\newtheorem{proposition}[theorem]{Proposition}
\newtheorem{example}[theorem]{Example}

\newtheorem{corollary}[theorem]{Corollary}

\newtheorem{maintheorem}{Theorem}

\title[The moduli space of maps from hypersurfaces and its Goodwillie tower]{The moduli space of maps from hypersurfaces\\and its Goodwillie tower}
\author{Alexis Aumonier}
\date{}
\address{Department of Mathematics, Stockholms universitet, 106 91 Stockholm, Sweden}
\email{alexis.aumonier@math.su.se}

\begin{document}

\begin{abstract}
We define a moduli space of continuous maps from smooth hypersurfaces varying in an ample linear system. To study its topology, we construct a scanning map to an appropriate section space and investigate its effect on homology. Our main innovation is a general argument based on Goodwillie calculus and applicable to various moduli spaces of maps. Along the way, we compute the Goodwillie tower of the section space functor based at any retractive space.
\end{abstract}

\maketitle

\setcounter{tocdepth}{1}
\tableofcontents

\section{Introduction}

Let $X$ be a smooth projective complex variety and $\cL$ be a very ample line bundle on it. Recall that the linear system $|\cL|$ contains an open dense locus $|\cL|^{\mathrm{ns}} \subset |\cL|$ parameterising the smooth hypersurfaces. Let $Y$ be a topological space. The present work is concerned with studying the moduli space of smooth hypersurfaces in $Y$, or in other words the space of maps
\[
    \set{(H,g)}{H \in |\cL|^{\mathrm{ns}}, \ g \colon H \to Y \text{ continuous map}}.
\]
This space lies at the interface between algebraic geometry and algebraic topology: on the one hand we are varying the hypersurfaces $H$ in an algebraic moduli space $|\cL|^{\mathrm{ns}}$, but on the other hand we only require the maps $g$ to be continuous. Our main motivation, on which we will expand shortly, is to further our understanding of $|\cL|^{\mathrm{ns}}$ via tools from homotopy theory. Our second goal with this paper is to showcase a general technique to study moduli spaces of continuous maps using Goodwillie calculus.

\bigskip

Let us now introduce some notation and state our main theorem. We will work in a slightly more general setting than above and allow $Y \to X$ to be a fibration. Let $\cM(\cL) \subset H^0(X,\cL)$ be the space of global algebraic sections of $\cL$ which are non-singular: $s \in \cM(\cL)$ if $(s,\mathrm{d}s)(x) \neq (0,0)$ for all $x \in X$. For any $s \in \cM(\cL)$, let
\[
    V(s) = \set{x \in X}{s(x) = 0} \subset X
\]
be the smooth hypersurface given by the vanishing locus of $s$. The main object of this article is the space
\[
    \cM(\cL;Y) = \set{(s,g)}{s \in \cM(\cL), \ g \in \Gamma(Y|_{V(s)})},
\]
where $\Gamma(-)$ denotes the space of continuous sections, and $Y|_{V(s)}$ the restriction of $Y$ to a hypersurface $V(s) \subset X$. When $Y = X \times Y' \to X$ is a trivial fibration, this construction generalises the moduli space of maps to $Y'$ as above. Let $\Omega^1_X$ be the cotangent bundle of $X$, let $*_X$ denote the fibrewise join of bundles on $X$, and let $S(-)$ denote the sphere bundle of a vector bundle. Our main theorem compares the cohomologies of $\cM(\cL;Y)$ and a continuous section space in a range of degrees:
\begin{maintheorem}\label{maintheorem:scanning}
Let $X$ be a smooth projective complex variety with a very ample line bundle $\cL$. Let $Y \to X$ be a fibration with a section, and assume that the fibre is $(\dim_\bR(X)-2)$-connected. If $\cL$ is $d$-jet ample, then there is a morphism of cohomology rings
\[
    H^*\big(\Gamma(S(\cL) *_X (S(\Omega^1_X \otimes \cL) \times_X Y)); \bZ\big) \lra H^*\big(\cM(\cL;Y);\bZ\big)
\]
which is an isomorphism in degrees $* < \frac{d-3}{2}$.
\end{maintheorem}
In \cref{theorem:main} below, we state a somewhat more refined version of our main theorem by first defining a \emph{scanning map} from (a fattened up version of) $\cM(\cL;Y)$ to a section space of a fibration equivalent to the one appearing in the theorem above. We also recall the definition of jet ampleness in \cref{definition:jet-ampleness}, but the reader is invited to just take $\cL$ to be the $d$th tensor power of a very ample line bundle.

\bigskip

When $Y = X$ is the trivial fibration, $\cM(\cL;Y) = \cM(\cL)$ is simply the moduli space of equations of smooth hypersurfaces in the linear system $|\cL|$. Its homology was studied in a range via an h-principle in \cite[Corollary~8.1]{hprinciple}. However the proof given in \cite[\S~2.1]{hprinciple}, or in any other typical proof of an h-principle à la Vassiliev \cite{Vassiliev}, relies on Alexander duality and studying a complementary moduli space of ``singular objects". (For example, $\cM(\cL)$ is studied via its complementary locus of singular sections of $\cL$.) Upon adding the data of continuous maps to $Y$, it is not clear that such a complementary space should exist. Instead we propose a novel approach using Goodwillie calculus.

To highlight the key ideas, let us be very informal and consider a generic moduli space $\cM$ of geometric objects, e.g. smooth hypersurfaces ($\cM(\cL)$), Riemann surfaces ($\cM_g$), configurations of points ($\Conf_i(X)$), etc. Suppose that $\cM$ is amenable to homotopical study via a ``scanning map" $\cM \to \cF$ which induces an isomorphism in homology in a range of degrees. In the existing literature, the space $\cF$ usually takes the form of a section space (see \cite{McDuff1975,Segal1973} for configuration spaces, \cite{hprinciple} for hypersurfaces, \cite{FedericoOscar} for smooth surfaces) or an infinite loop space (see \cite{MadsenWeiss,OscarSorenActa,OscarSorenAnnals,OscarSorenJAMS} for the seminal works on moduli of manifolds). Let $Y$ be a space, and let $\cM(Y)$ be a moduli space of continuous maps from objects in $\cM$ to $Y$. One can naturally wonder if there is a space $\cF(Y)$ and a generalisation of the scanning map $\cM(Y) \to \cF(Y)$ which is still highly connected in homology. Here are two examples:
\begin{itemize}
    \item Let $M$ be a connected open manifold which is the interior of a compact manifold with boundary and has dimension greater than 1. Let $\UConf_n(M)$ denote the unordered configuration space of $n$ points on $M$. Let $S^{TM}$ be the fibrewise 1-point compactified tangent bundle of $M$. Then there is a scanning map
    \[
        \UConf_n(M) \lra \Gamma_c(S^{TM})
    \]
    to the space of compactly supported sections, which induces an isomorphism in homology in degrees $* \ll n$ onto the component that it hits. See \cite[Theorem~1.2]{McDuff1975} for the statement and \cite{Segal1973} for the origin of this scanning map. Let $Y$ be a connected space and let 
    \[
        \UConf_n(M;Y) = \set{(m_1,\dotsc,m_n) \in M^n}{m_i \neq m_j, \ i \neq j} \times_{\Sigma_n} Y^n
    \]
    be the unordered configuration space of $n$ particles in $M$ labelled in $Y$. Then there is a map
    \[
        \UConf_n(M;Y) \lra \Gamma_c(S^{TM} \wedge Y_+)
    \]
    which induces an isomorphism in homology in degrees $* \ll n$ onto the component that it hits. This follows from combining \cite[Proposition~2]{Bodigheimer}, homological stability \cite[Theorem~A]{OscarConfig}, and the group completion theorem.
    \item Let $\cM_g$ be the moduli space of compact Riemann surfaces of genus $g$. Let $\MTSO(2)$ be the Thom spectrum of the virtual inverse of the tautological bundle on $BSO(2)$. There is a scanning map \cite[p.~511]{MadsenTillmann}
    \[
        \cM_g \lra \Omega^\infty \MTSO(2)
    \]
    which induces an isomorphism in homology in degrees $* \ll g$ onto the component that it hits \cite[Theorem~1.1]{MadsenWeiss}. Let $Y$ be a simply connected space and let
    \[
        \cM_g(Y) = \set{(C,f)}{C \in \cM_g, \ f \colon C \to Y}
    \]
    be the moduli space of maps to $Y$ (see \cite[Introduction]{CohenMadsen} for a precise definition). There is a map
    \[
        \cM_g(Y) \lra \Omega^\infty (\MTSO(2) \wedge Y_+)
    \]
    which induces an isomorphism in homology in degrees $* \ll g$ onto the component that it hits, see \cite[Theorem~0.1]{CohenMadsen} for the original statement and \cite[Theorem~1.2 and \S~1.5]{OscarResolutionmoduli} for a generalisation to any tangential structure.
\end{itemize}
Let us assume that we have found a good candidate $\cF(Y)$ with a scanning map $\cM(Y) \to \cF(Y)$. To check that it is the right guess, we may proceed as follows. If they are natural enough, the scanning maps for varying $Y$'s assemble into a natural transformation of functors on the category of spaces. In fact, as we are ultimately interested in homology, it is technically easier to work with the suspended functors
\[
    \Sigma^\infty_+ \cM(Y) \lra \Sigma^\infty_+\cF(Y)
\]
from the category of spaces to spectra. This transformation induces a map on the Goodwillie towers
\[
\begin{tikzcd}
\cdots \arrow[r] & P_n\Sigma^\infty_+\cM(Y) \arrow[r] \arrow[d] & \cdots \arrow[r] & P_1\Sigma^\infty_+\cM(Y) \arrow[r] \arrow[d] & P_0\Sigma^\infty_+\cM(Y) \simeq \Sigma^\infty_+\cM \arrow[d] \\
\cdots \arrow[r] & P_n\Sigma^\infty_+\cF(Y) \arrow[r]           & \cdots \arrow[r] & P_1\Sigma^\infty_+\cF(Y) \arrow[r]           & P_0\Sigma^\infty_+\cF(Y) \simeq \Sigma^\infty_+\cF         
\end{tikzcd}
\]
In many situations (usually when the connectivity of $Y$ is greater than the dimension of the geometric objects parameterised by $\cM$), the inverse limit of the tower recovers the functor. In that case, to show that $\Sigma^\infty_+ \cM(Y) \to \Sigma^\infty_+\cF(Y)$ is highly connected, it is enough to prove that
\[
    P_n\Sigma^\infty_+\cM(Y) \lra P_n\Sigma^\infty_+\cF(Y)
\]
is highly connected for each $n \geq 0$. By induction, it then suffices to show that this is the case for the induced maps on the layers $D_n(-) = \fib(P_n(-) \to P_{n-1}(-))$ of the towers:
\[
    D_n\Sigma^\infty_+\cM(Y) \lra D_n\Sigma^\infty_+\cF(Y).
\]
One of the major practical aspects of Goodwillie calculus is that these layers are often computable. In fact, if the moduli space $\cM(Y)$ sits in a fibration
\[
    \Map(M,Y) \lra \cM(Y) \lra \cM
\]
then formal arguments can be made to compute the derivatives of $\Sigma^\infty_+\cM(Y)$ from those of $\Sigma^\infty_+\Map(M,Y)$ for a fixed $M \in \cM$. For the latter, when $M$ is a manifold, Arone \cite[Theorem~2]{Greg} has computed the $n$th derivative in terms of the configuration space of $n$ points on $M$. It follows that in good cases $D_n\Sigma^\infty_+\cM(Y)$ is expressed in terms of a moduli space $\cM_n$ parameterising objects in $\cM$ with $n$ distinct markings, e.g. bicoloured configurations when $\cM = \Conf_i(X)$ or marked Riemann surfaces when $\cM = \cM_g$. On the other hand, if the candidate functor $\cF(Y)$ is defined via natural constructions in homotopy theory, one can usually compute its Goodwillie derivatives. We have explained the simplest case of infinite loop spaces (with $\cM = \cM_g$) in \cref{section:POC}. For section spaces, e.g. in this article or for configuration spaces, one can instead rely on the results of Arone or their generalisation to section spaces established in the present work:
\begin{maintheorem}\label{maintheorem:goodwillie}
Let $X$ be a finite CW-complex. Let $E \in \catS_{X//X}$ be a retractive space over $X$. The Goodwillie tower of the functor
\begin{align*}
    F \colon \catS_{E//E} &\lra \catSp \\
        Y &\longmapsto \Sigma^\infty\Gamma(X;Y)
\end{align*}
converges for retractive spaces $Y \in \catS_{E//E}$ such that the projection $Y \to E$ is at least $(\dim(X)+1)$-connected. Furthermore, the $n$th polynomial approximation is given by the functor $A_n(X,-)$ constructed in \cref{definition:An}. When $X$ is a closed manifold, the $n$th derivative is given by the parameterised spectrum
\[
    \mathrm{ev}_!\bS^0 \otimes p^*\bS^{-TX^n} \otimes p^*\omega_!\bS^0 \in \catSp^{E^n \sslash \Sigma_n}
\]
where $p \colon E^n \sslash \Sigma_n \to X^n \sslash \Sigma_n$ is the projection, $\mathrm{ev} \colon X^n \sslash \Sigma_n \times \Gamma(E) \to E^n \sslash \Sigma_n$ is the evaluation map, and $\omega \colon \UConf_n(X) \to X^n \sslash \Sigma_n$ is the inclusion of the unordered configuration space.
\end{maintheorem}
We prove this as \cref{theorem:goodwillie-approximation} in the body of this article. Although slightly cumbersome to write down, the Goodwillie polynomial approximations are made very explicit in \cref{definition:An}. For mapping spaces, this theorem was obtained by Arone in \cite[Theorem~2]{Greg}. For the section space functor $\catS_{X//X} \ni Y \mapsto \Sigma^\infty \Gamma(X;Y)$ based at the ``origin" $X \in \catS_{X//X}$ the theorem follows from work of Malkiewich \cite[\S~3]{Malkiewich} as explained in \cite[\S~4.4]{weissderivatives}. Our original contribution lies in allowing us to base the Goodwillie tower at \emph{any} space $E \in \catS_{X//X}$, and simplifying Arone's proof using the bifunctoriality of section spaces.

\subsection{Motivations and applications}

\subsubsection{Twisted coefficients}
Recall that $\cM(\cL) \subset H^0(X,\cL)$ is the space of non-singular global sections of a very ample line bundle $\cL$ on a smooth projective complex variety $X$. Let 
\[
    \cZ_1(\cL) = \set{(s,x) \in \cM(\cL) \times X}{s(x) = 0}
\]
be the incidence variety. It is a fibre bundle $\cZ_1(\cL) \to \cM(\cL)$ with fibre over $s$ given by the smooth hypersurface $V(s)$. In particular, the monodromy of this bundle gives a local system on $\cM(\cL)$:
\[
    \pi_1(\cM(\cL);s) \lra \operatorname{Aut} H^*(V(s);\bZ).
\]
The fundamental group of $\cM(\cL)$ is generally quite complicated, see \cite{Lonne} for a presentation when $X = \bP^n$. It is nonetheless interesting and natural to ask for a computation of the cohomology of $\cM(\cL)$ with coefficients twisted in the local system above:
\[
    H^*\left(\cM(\cL); \cH^*(V(s);\bQ)^{\otimes n}\right) \cong \ ?
\]
The trivial case ($n = 0$) was computed in a range of degrees in \cite[Corollary~8.1]{hprinciple}. The simplest twisted case ($n = 1$) was treated in a range of degrees in \cite[\S~1.1]{withronno} using the isomorphism
\[
    H^*(\cM(\cL); \cH^*(V(s);\bQ)) \cong H^*(\cZ_1(\cL);\bQ).
\]
For higher tensor powers ($n \geq 2$), and more generally for coefficients twisted by any Schur functor, one can instead rely on a technique of Randal-Williams \cite[\S~5.3]{OscarTwisted}. Roughly, this technique first requires a computation of the cohomology of the moduli of maps to an Eilenberg--MacLane space $K(\bQ,\ell)$, uses
\[
    H^*(\Map(V(s), K(\bQ,\ell)); \bQ) \cong \Sym \left(\bigoplus_{i > 0} H^{\ell -i}(V(s);\bQ)^\vee[i]\right),
\]
and then exploits the Serre spectral sequence of the projection $\cM(\cL; K(\bQ,\ell)) \to \cM(\cL)$. Explicit computations, following this technique and relying on \cref{maintheorem:scanning}, will appear elsewhere in forthcoming work of the author and Ronno Das.

\subsubsection{Families of maps to projective space}

Let $\bP^m$ denote a complex projective space. Let $B$ be a space. A map $f \colon B \to \bP^m$ is said to have degree $\deg f = f^*(h) \in H^2(B;\bZ)$, where $h \in H^2(\bP^m;\bZ) \cong \bZ$ is the positive generator. We write $\Map_\alpha(B,\bP^m) \subset \Map(B,\bP^m)$ for the subspace of maps of degree $\alpha \in H^2(B;\bZ)$. Let $X$ be a smooth projective variety of (complex) dimension at least $4$, and $\cL$ be a very ample line bundle on $X$. Let $\alpha \in H^2(X;\bZ)$. By the Lefschetz hyperplane theorem $H^2(V(s);\bZ) \cong H^2(X;\bZ)$ for any non-singular section $s \in \cM(\cL)$, hence we may speak of maps $V(s) \to \bP^m$ of degree $\alpha$. In this setting, there are two natural moduli spaces of maps to $\bP^m$. The first one is the moduli space of continuous maps of degree $\alpha$:
\[
    \cM_\alpha(\cL; \bP^m) = \set{(s,g)}{s \in \cM(\cL), \ g \in \Map_\alpha(V(s),\bP^m)}.
\]
The second is the moduli space of algebraic (equivalently holomorphic) maps:
\[
    \cM_\alpha^\mathrm{alg}(\cL; \bP^m) = \set{(s,g)}{s \in \cM(\cL), \ g \in \Hol_\alpha(V(s),\bP^m)}.
\]
This latter space was studied in \cite{familiesofmaps} where it is shown that the natural inclusion
\[
    \cM_\alpha^\mathrm{alg}(\cL; \bP^m) \hookrightarrow \cM_\alpha(\cL; \bP^m)
\]
induces an isomorphism in homology in a range of degrees depending on $\alpha$, see \cite[Main Theorem]{familiesofmaps} and \cite[Examples~3.7 and~3.12]{familiesofmaps}.
\begin{theorem}\label{theorem:familiesofmaps}
Let $X$ be a smooth projective complex variety of dimension at least 4 with first cohomology $H^1(X;\bZ) = 0$. Let $\cL$ be a $d$-jet ample line bundle on $X$. Let $\alpha \in H^2(X;\bZ)$ and choose a complex line bundle $\cT_\alpha$ on $X$ with first Chern class $c_1(\cT_\alpha) = \alpha$. Then, for any integer $m \geq \dim X$, there is an isomorphism of cohomology rings
\[
    H^*(\cM_\alpha^\mathrm{alg}(\cL; \bP^m);\bZ) \cong H^*(\Gamma(S(\cL) *_X (S(\Omega^1_X \otimes \cL) \times_X S(\cT_\alpha^{\oplus m+1})))\sslash S^1; \bZ)
\]
in degrees $* < \min\big( \frac{d-3}{2}, (2m-2\dim X + 3)d(\alpha) + 2m-2\dim X -1\big)$, where $d(\alpha)$ is as in \cite[Definition~3.5]{familiesofmaps} and $S^1 \subset \bC^\times$ acts on the section space by the scalar action on $\cT_\alpha^{\oplus m+1}$.
\end{theorem}
The number $d(\alpha)$ is hard to compute in general. However, by \cite[Example~3.7]{familiesofmaps} it is at least equal to the maximal $k$ such that all algebraic line bundles on $X$ with first Chern class $\alpha$ are $k$-very ample. For the easiest example, consider $X = \bP^n$ and $\alpha \in H^2(\bP^n;\bZ)$. Then $d(\alpha)$ is at least $\deg \alpha$, with $\deg \colon H^2(\bP^n;\bZ) \cong \bZ$ the degree isomorphism.

\begin{proof}[Proof of \cref{theorem:familiesofmaps}]
\Cref{maintheorem:scanning} applies to the moduli space $\cM(\cL,S(\cT_\alpha^{\oplus m+1}))$ and shows that a scanning map
\[
    \cM(\cL,S(\cT_\alpha^{\oplus m+1})) \lra \Gamma(S(\cL) *_X (S(\Omega^1_X \otimes \cL) \times_X S(\cT_\alpha^{\oplus m+1})))
\]
induces an isomorphism in integral homology in the range of degrees $* < \frac{d-5}{2}$. There is a principal fibration
\[
    \Map(V(s),S^1) \lra \Gammacon(V(s); S(\cT_\alpha^{\oplus m+1})) \lra \Map_\alpha(V(s), \bP^m)
\]
for any hypersurface $V(s)$, see e.g. \cite[\S~4.2]{familiesofmaps}. By assumption $H^1(V(s);\bZ) = H^1(X;\bZ) = 0$, so $\Map(V(s),S^1) \simeq S^1$. Hence
\[
    \cM_\alpha(\cL;\bP^m) \simeq \cM(\cL,S(\cT_\alpha^{\oplus m+1})) \sslash S^1.
\]
The theorem then follows by \cite[Main Theorem]{familiesofmaps} which compares the homologies of $\cM_\alpha(\cL;\bP^m)$ and $\cM_\alpha^\mathrm{alg}(\cL; \bP^m)$ in a range of degrees.
\end{proof}

\subsection{Outline}
In \cref{section:definitions} we give the formal definition of the moduli space of sections, define the scanning map, and precisely state our main theorem. Then in \cref{section:goodwillie} we fully compute the Goodwillie tower of section spaces based anywhere. This section is purely in the realm of homotopy theory and the reader can either jump straight to it (if only interested in \cref{maintheorem:goodwillie}) or take it as a black box (if only concerned with its application to moduli spaces). We prove our main theorem in \cref{section:comparison,section:equivalence}. We finish with a short proof of concept of our method in \cref{section:POC}.

\subsection*{Shortcomings}
The author would like to apologise for the use of point-set models in the arguments. It seems that clean $\infty$-categorical formulas for the Goodwillie tower of section spaces should be possible to write down. However, the author was not successful in comparing the abstract formulas with the more concrete topological spaces that are the subject of this article. A middle-ground was chosen for concreteness, but some arguments are unfortunately cumbersome.

\subsection*{Acknowledgments}
I am happy to thank Greg Arone for several helpful discussions around Goodwillie calculus for mapping spaces, and Oscar Harr for technical help with sheaves. Thanks to Ronno Das for an ongoing collaboration which influenced the writing of this article. This article was written in January 2026 and I gave talks about it in Utrecht and Lisbon. I thank the audiences for some helpful comments. I was supported by Dan Petersen's Wallenberg Scholar fellowship. OpenAI's LLM was used for final proofreading.

\section{The moduli space of maps from smooth hypersurfaces}\label{section:definitions}

\subsection{Recollections}
Let $X$ be a smooth projective complex variety, $\cL$ be a very ample line bundle on $X$. We will write $\Gammahol(-)$, resp. $\Gamma(-)$, for the space of global holomorphic, resp. continuous, sections, and $\mathrm{d}(-)$ for the derivative of a section. Let
\begin{equation}\label{eqn:def-moduli}
    \cM(\cL) = \set{s \in \Gammahol(X,\cL)}{(s,\mathrm{d}s)(x) \neq (0,0) \ \forall x \in X}
\end{equation}
be the moduli of (equations\footnote{It might be more geometrically meaningful to consider the locus of smooth hypersurfaces $\cM(\cL) / \bC^\times \subset |\cL|$ but it is technically preferable to work with the space of sections. However the comparison map in \cref{theorem:main} is $\bC^\times$-equivariant (choosing $\varepsilon(s)$ to be equivariant), so it suffices to pass to orbits in that statement to study the quotient $\cM(\cL) / \bC^\times$.} of) smooth hypersurfaces in the linear system $|\cL|$. We shall now recall how its topology can be studied using an h-principle. Let $\Omega^1_X$ denote the cotangent bundle of $X$. The first jet bundle of $\cL$ is an algebraic vector bundle $J^1\cL$ on $X$, whose underlying smooth complex vector bundle is isomorphic to the direct sum
\[
    J^1\cL \cong \cL \oplus (\Omega^1_X \otimes \cL).
\]
We refer to \cite[IV.16.7]{EGA4} for a construction. Upon taking global holomorphic sections and choosing a smooth connection, one defines the jet map:
\begin{align*}
    j^1 \colon \Gammahol(\cL) &\lra \Gamma(J^1\cL) \cong \Gamma(\cL) \oplus \Gamma(\Omega^1_X \otimes \cL) \\
        s &\longmapsto (s,\mathrm{d}s).
\end{align*}
To give ranges in our results, we recall the definition of jet ampleness of a line bundle:
\begin{definition}[Compare \cite{jetampleness}]\label{definition:jet-ampleness}
Let $k \geq 0$ be an integer. Let $x_1, \ldots, x_t$ be $t$ distinct points in $X$ and $(k_1,\ldots,k_t)$ be a $t$-tuple of positive integers such that $\sum_i k_i = k+1$. Denote by $\cO$ the structure sheaf of $X$ and by $\fm_i$ the maximal ideal sheaf corresponding to $x_i$. We regard the tensor product $\otimes_{i=1}^t \fm_i^{k_i}$ as a subsheaf of $\cO$ under the multiplication map $\otimes_{i=1}^t \fm_i^{k_i} \to \cO$. We say that $\cL$ is \emph{$k$-jet ample} if the evaluation map
\[
    \Gammahol\left(\cL\right) \lra \bigoplus_{i=1}^t \Gammahol\left(\cL \otimes \left(\cO / \fm_i^{k_i}\right)\right)
\]
is surjective for any $x_1,\ldots,x_t$ and $k_1,\ldots,k_t$ as above.
\end{definition}
\begin{example}
If $\cL$ is very ample, then the $k$th tensor power $\cL^{\otimes k}$ is $k$-jet ample, see~\cite[Proposition 2.3]{jetampleness}.
\end{example}

To motivate what will follow, we recall one of the main results of \cite{hprinciple} which studies the effect of the jet map on homology:
\begin{theorem}[{\cite[Corollary~8.1]{hprinciple}}]\label{theorem:old-hprinciple}
Suppose that $\cL$ is $d$-jet ample. Then the jet map
\[
    j^1 \colon \cM(\cL) \lra \Gamma(J^1\cL - 0)
\]
induces an isomorphism in integral homology in the range of degrees $* < \frac{d-1}{2}$.
\end{theorem}
The point of that theorem resides in the fact that the continuous section space $\Gamma(J^1\cL - 0)$ is amenable to algebraic topology. We will now work towards adding the data of continuous maps to a background space, and stating the proper generalisation of the theorem above.

\subsection{The moduli of maps}

Let $q \colon Y \to X$ be a fibration. The main object under consideration in this article is
\begin{equation}\label{eqn:def-moduli-Y}
    \cM(\cL;Y) = \set{(s,g)}{s \in \cM(\cL), \ g \in \Gamma(Y|_{V(s)})}
\end{equation}
which we call the \emph{moduli of smooth hypersurfaces equipped with a section of $Y$}. To describe its topology, recall that we have a smooth fibre bundle (a direct consequence of Ehresmann's fibration theorem): 
\[
    \cZ_1(\cL) = \set{(s,x) \in \cM(\cL) \times X}{s(x) = 0} \lra \cM(\cL), \quad (s,x) \longmapsto s.
\]
Its local triviality can be used to topologise $\cM(\cL;Y)$ as follows: let $s \in U \subset \cM(\cL)$ be a small open neighbourhood of some $s$ over which the bundle $\cZ_1(\cL) \to \cM(\cL)$ is trivial. Then we have a homeomorphism $\theta \colon U \times V(s) \cong \cZ_1(\cL)|_U$, and we declare the induced bijection
\begin{align*}
    \cM(\cL;Y)|_U &\overset{\cong}{\lra} \set{(u,g) \in U \times \Map(V(s),Y)}{q \circ g = \tilde\theta(u,-)} \\
    (u,g) &\longmapsto (u,g \circ \theta(u,-))
\end{align*}
to be a homeomorphism, where the right-hand side is equipped with the subspace topology (the mapping space has the compact-open topology), and $\tilde\theta$ is the composition of $\theta$ and the inclusion $\cZ_1(\cL) \subset \cM(\cL) \times X$. By gluing local patches over a trivialising open cover we topologise the full moduli space.

\subsection{Scanning map and the main theorem}\label{subsection:scanning}

We pick once and for all a Hermitian metric on $J^1\cL$ and write $S(J^1\cL)$ for the unit sphere bundle. Let
\begin{equation*}
    \cM(\cL) \times X \lra S(J^1\cL) \ , \quad (s,x) \longmapsto j^1s(x)
\end{equation*}
be the jet map evaluation. Here, and everywhere else in this paper, we implicitly rescale any vector in $J^1\cL -0$ so that it lies on the unit sphere bundle. Given $(s,x) \in \cM(\cL) \times X$, if $s(x) = 0$ then the derivative $\mathrm{d}s(x) \neq 0$. Hence the jet evaluation map is transverse to $S(\Omega^1_X \otimes \cL) \subset S(J^1\cL)$. Let $N$ be the closed tubular neighbourhood of $S(\Omega^1_X \otimes \cL)$ defined by
\[
    N = \set{(v,w) \in S(\cL \oplus \Omega^1_X \otimes \cL)}{\lVert v \rVert \leq \lVert w \rVert},
\]
using a fixed splitting $\cL \oplus \Omega^1_X \otimes \cL \cong J^1\cL$ and writing $\lVert-\rVert$ for the fixed metric. By compactness of $X$, for any $s \in \cM(\cL)$ there exists a small enough $0 < \varepsilon(s) < 1$ such that
\[
    \tub V(s) = \set{x \in X}{\lVert s(x) \rVert \leq \varepsilon(s) \lVert \mathrm{d}s(x) \rVert}
\]
is a closed tubular neighbourhood of $V(s)$. We may choose $\varepsilon(s)$ to depend continuously on $s$, and define the rescaled jet evaluation map:
\begin{equation}\label{eqn:jet-evaluation-map}
    \Psi \colon \cM(\cL) \times X \lra S(J^1\cL) \ , \quad (s,x) \longmapsto \left( \frac{s(x)}{\varepsilon(s)}, \mathrm{d}s(x) \right) \bigg/ \left\lVert \left( \frac{s(x)}{\varepsilon(s)}, \mathrm{d}s(x) \right) \right\rVert.
\end{equation}
By construction, for any $s \in \cM(\cL)$, the intersection
\[
    (\{s\} \times X) \cap \Psi^{-1}(N)
\]
is the closed tubular neighbourhood $\tub V(s)$ of $V(s)$ in $X$. We will need to slightly enlarge our moduli space $\cM(\cL;Y)$, so we define
\[
    \cM^\mathrm{tub}(\cL;Y) = \set{(s,g)}{s \in \cM(\cL), \ g \in \Gamma(Y|_{\tub V(s)})},
\]
and topologise it similarly to $\cM(\cL;Y)$. By shrinking the tubular neighbourhood onto its zero section, we obtain a weak homotopy equivalence
\[
    \cM^\mathrm{tub}(\cL;Y) \overset{\simeq}{\lra} \cM(\cL;Y).
\]
Likewise, let
\[
    \pi \colon \cZ^\mathrm{tub}_1(\cL) = \set{(s, x) \in \cM(\cL) \times X}{x \in \tub V(s)} \lra \cM(\cL)
\]
be the fattened up universal smooth hypersurface bundle.

\bigskip

Let $S^\cL \to X$ be the sphere bundle obtained by fibrewise one-point compactifying the bundle $\cL$. We write $\partial N = N - \mathring{N}$ for the sphere bundle of the tubular neighbourhood, and $\rho \colon N \to X$ for the projection. To construct the comparison map in our main theorem, we first define a section of $\rho^*S^\cL \to N$. Let $\nu \colon \mathring{N} \to S(\Omega^1_X \otimes \cL)$ be the normal bundle projection map, and $p \colon S(J^1\cL) \to X$ be the projection. The isomorphism $J^1\cL \cong \cL \oplus (\Omega^1_X \otimes \cL) $ implies that the normal bundle $\mathring{N}$ is isomorphic to $(p^*\cL)|_{S(\Omega^1_X \otimes \cL)}$. Hence
\[
    \rho|_{\mathring{N}}^*\cL \cong (p \circ \nu)^*\cL \cong \nu^*\left((p^*\cL)|_{S(\Omega^1_X \otimes \cL)}\right) \cong \nu^*\mathring{N} \cong \mathring{N} \times_{S(\Omega^1_X \otimes \cL)} \mathring{N}.
\]
The diagonal $\mathring{N} \to \mathring{N} \times_{S(\Omega^1_X \otimes \cL)} \mathring{N}$ thus yields a section of $\rho|_{\mathring{N}}^*\cL$. We extend it to a section $N \to \rho^*S^{\cL}$ by sending the boundary $\partial N$ to the basepoint at infinity. 

Because the jet evaluation map~\eqref{eqn:jet-evaluation-map} sends the subspace $\cZ^\mathrm{tub}_1(\cL) \subset \cM(\cL) \times X$ into $N$, we obtain a pullback square:
\[
\begin{tikzcd}
|[label={[label distance=-2mm]-45:\lrcorner}]| \pi^*S^\cL \arrow[d] \arrow[r]            & \rho^*S^\cL \arrow[d] \\
\cZ^\mathrm{tub}_1(\cL) \arrow[r, "\Psi"] & N                    
\end{tikzcd}
\]
In particular, we may pull back the section defined above to a section
\begin{equation}\label{eqn:can-section}
    \can \colon \cZ^\mathrm{tub}_1(\cL) \lra \pi^*S^\cL.
\end{equation}
Informally, this canonical section associates to a pair $(s,x)$ the unique vector $\vec{v}$ corresponding to $x$ in the open tubular neighbourhood of $V(s)$, or the basepoint if $x \in \partial \tub V(s)$, see \cref{figure:tubular-vector} for an illustration of possible $\vec{v}$.

\begin{figure}
\[
\begin{tikzpicture}[scale=0.8]

\draw[red] (0,0) .. controls (0,2) and (4,2) .. (5,5);
\draw[blue] (2,-1) .. controls (2,1) and (6,1) .. (7,4);
\draw[blue] (-2,1) .. controls (-2,3) and (2,3) .. (3,6);

\draw [blue,decorate,decoration={brace,amplitude=20pt,raise=1ex}] (3,6) -- (7,4) node[midway,yshift=3em,xshift=3em]{$\tub V(s)$};
\node[red] at (0,-0.3) {$V(s)$};

\node at (4.5,4) {$\bullet$};
\node at (5.15,4) {$\vec{v} = 0$};

\node at (1.97,3.05) {$\bullet$};
\draw [->] (2.45,2.31) -- (2,3) node[right, midway, yshift=0.5em]{$\vec{v}$};

\node at (2.44,0) {$\bullet$};
\draw[->] (1.3,1.66) -- (2.41,0.05) node[right, midway]{$\vec{v} = \infty$};

\end{tikzpicture}
\]
\caption{Possible values of the section $\can$ restricted to some $\{s\} \times \tub V(s) \subset \cZ^\mathrm{tub}_1(\cL)$.}
\label{figure:tubular-vector}
\end{figure}

\bigskip

Let us now assume that the fibration $Y \to X$ has a section. One of the main insights of this article is the definition of the following ``scanning" map:
\begin{align}\label{eqn:scanning-map}
\begin{split}
    X \times \cM^\mathrm{tub}(\cL;Y) &\lra (N \sqcup S(J^1\cL)) \wedge_{S(J^1\cL)} p^*(S^\cL \wedge_X Y) \\
        x,(s,g) &\longmapsto  \begin{cases}
                                \Psi(s,x) \wedge \can(s,x) \wedge g(x) &\text{if } x \in \tub V(s) \\
                                \text{fibrewise basepoint} & \text{if } x \not\in \tub V(s).
                              \end{cases}
\end{split}
\end{align}
Let us comment on our notation: we write $- \wedge_{S(J^1\cL)}-$ for the fibrewise smash product of retractive spaces over $S(J^1\cL)$, see $N \sqcup S(J^1\cL) \to S(J^1\cL)$ as a retractive space via the section on the second summand, and see $S^\cL$ being pointed at infinity. When $x \in \tub V(s)$, we have $\Psi(s,x) \in N \subset N \sqcup S(J^1\cL)$ (i.e. not the fibrewise basepoint). By design, $\can(s,x) \in (\pi^*S^\cL)|_{(s,x)} = S^\cL|_x = (p^*S^\cL)|_{\Psi(s,x)}$ so that our map is well-defined. Continuity is ensured by the fact that $\can$ maps the boundary $\partial \tub V(s)$ to the fibrewise basepoints. We are now ready to state our main theorem:
\begin{theorem}\label{theorem:main}
Let $Y \to X$ be a well-sectioned fibration. Suppose that $\cL$ is $d$-jet ample, and that $Y \to X$ is at least $(\dim_\bR(X) - 1)$-connected. Then the adjoint of the scanning map~\eqref{eqn:scanning-map}
\begin{equation}\label{eqn:main-map}
    \cM(\cL;Y) \simeq \cM^\mathrm{tub}(\cL;Y) \lra \Gamma\left(X; (N \sqcup S(J^1\cL)) \wedge_{S(J^1\cL)} p^*(S^\cL \wedge_X Y)\right)
\end{equation}
induces an isomorphism in homology in the range of degrees $* < \frac{d-3}{2}$ and a surjection in degrees $* \leq \frac{d-3}{2}$.
\end{theorem}
In that statement, we only assume that $Y \to X$ is well-sectioned (i.e. that it is an ex-fibration in the sense of \cite[Definition~8.1.1]{MaySigurdsson}) to guarantee the homotopy invariance of our formulas.

\begin{remark}
Even though $N \hookrightarrow S(J^1\cL)$ is not a fibration, the composition $\rho \colon N \hookrightarrow S(J^1\cL) \to X$ is a fibration. Hence $(N \sqcup S(J^1\cL)) \wedge_{S(J^1\cL)} p^*(S^\cL \wedge_X Y) \to X$ is also a fibration if $Y \to X$ is well-sectioned \cite[Proposition~8.2.1.(ii)]{MaySigurdsson}. In particular, the homotopy type of the section space of \cref{theorem:main} only depends on the homotopy type of that fibration. We will prove later in \cref{proposition:ansatz-is-join} that this ansatz formula can be replaced by the fibrewise join appearing in \cref{maintheorem:scanning}.
\end{remark}

\begin{remark}
We briefly comment on the range of degrees of \cref{theorem:main}. We do not believe it to be optimal, as it relies on inputs which are probably not optimal: the ranges appearing in \cref{theorem:old-hprinciple} and \cite[Theorem~B]{withronno}, and the jet-ampleness condition that could be replaced by a weaker assumption only on jets of order $\leq 1$.
\end{remark}

\section{Goodwillie tower of section spaces}\label{section:goodwillie}

\subsection{Recollections and conventions}\label{subsection:recollections}
Let us begin with a bit of notation that we shall use throughout this article.

\subsubsection{Categories and parameterised objects}
We denote the $\infty$-category of spaces by $\catS$, and the $\infty$-category of spectra by $\catSp$. When engaging in formal manipulations we will always work in the ``derived sense": pullbacks/pushouts/colimits/... are always the ``homotopy" versions, and we will add the word ``strict" if necessary to stress when we make point-set constructions. If $B \in \catS$ is a space, we write $\catS_{/B}$ for the $\infty$-category of spaces over $B$, and $\catS_{B//B}$ for the $\infty$-category of retractive spaces over $B$. The symmetric monoidal $\infty$-category of parameterised spectra over $B$ is denoted $\catSp^B = \mathsf{Fun}(B, \catSp)$. We denote the $k$-fold suspension-loop adjunction by
\[
    \Sigma^k_B \colon \catS_{B//B} \rightleftharpoons \catS_{B//B} \colon \Omega^k_B,
\]
and do similarly for the infinite suspension-loop adjunction. If $f \colon A \to B$ is a map of spaces, the pullback functors $f^* \colon \catSp^B \to \catSp^A$ and $f^* \colon \catS_{B//B} \to \catS_{A//A}$ have both a left and a right adjoint denoted $f_!$ and $f_*$ respectively. When $f \colon A \to *$ is the unique map to a point, we shall equivalently write
\[
    f_!(S) \simeq \colim_A S
\]
when applied to a parameterised spectrum $S \in \catSp^A$. Recall that by convention all our functors are derived: for example, if $f \colon A \to B$ and $S^0$ denotes the constant retractive space $S^0 \times A \to A$ then $f_!S^0$ is a retractive space over $B$ whose fibre at a point $b \in B$ is computed as the \emph{homotopy} fibre of $f$ at $b$ (with an added basepoint). We shall often use the projection formula $f_!(S\otimes f^*T) \simeq f_!(S) \otimes T$ for a general $f$. For a vector bundle $V \to A$, we write $S^V$ for the pointed (at infinity) sphere bundle obtained by fibrewise one-point compactifying. Its fibrewise suspension spectrum is denoted $\bS^V$, and this notation extends to virtual vector bundles, e.g. $\bS^{-V}$. Such spectra will typically appear when using the Costenoble--Waner duality isomorphism $f_*(-) \simeq f_!(- \otimes \bS^{-TA})$, with $f \colon A \to *$ and $A$ a closed manifold with tangent bundle $TA$. We refer to \cite[Appendix~A]{Markus} for proofs of all these facts.

If $S \in \catSp^A$ and $T \in \catSp^B$ are parameterised spectra over $A$ and $B$, their external tensor product $S \boxtimes T \in \catSp^{A \times B}$ is parameterised over the product $A \times B$: it is informally given by fibrewise taking tensor products. The tensor power $S^{\boxtimes k}$ is parameterised over $A^k$ and is equipped with the permutation action of the symmetric group $\Sigma_k$, hence can be refined to a parameterised spectrum over $A^k \sslash \Sigma_k$. We shall use the same notation for its refinement.

\subsubsection{Point-set conventions}
Let us also say a word about our category of topological spaces $\catTop$. We will always assume that a topological space is compactly generated and weakly Hausdorff. When working with retractive spaces, we shall often use the convenient notion of ex-fibrations from \cite[Definition~8.1.1]{MaySigurdsson}. We stress that we only appeal to these point-set technicalities so that our formulas have the correct homotopical properties, but any reader is invited to work with their favourite model categories as long as they present the same $\infty$-categories.

Secondly, we will often construct (parameterised) spectra via sequential spectra \cite[Definition~5.1.1]{MalkiewichLowTech}. The main point for us is that parameterised sequential spectra over $B$ with their levelwise model structure \cite[\S~5.2]{MalkiewichLowTech} present the infinity category $\catSp^B$, see \cite[Theorem~3.6]{ABG} and the paragraph above that theorem. We will mostly use this observation to compute the homotopy fibre of a levelwise fibration of sequential spectra as the strict fibre.

Thirdly, given a map $f \colon A \to B$ of topological spaces, we can define a sheaf of topological spaces on $B$ by the rule
\[
    B \overset{\text{open}}{\supset} U \longmapsto \Gamma(U; f^{-1}(U)),
\]
where the section space is equipped with the k-ification of its compact-open topology. Suppose now that $B$ is locally of singular shape \cite[Definition~A.4.15]{HA}, e.g. $B$ is a Euclidean neighbourhood retract (ENR), and has the homotopy type of a finite CW-complex. Then, by combining \cite[Theorem~3]{Milnor1959} and \cite[Example~A.4.12]{HA}, we see that the section space $\Gamma(U; f^{-1}(U))$ is weakly homotopy equivalent to the CW-complex $|\operatorname{Sing}\Gamma(U; f^{-1}(U))|$. Assuming furthermore that $B$ admits a basis of its topology which is stable under finite intersection and such that every basis element is paracompact, \cite[Proposition~7.1.3.14]{HTT} shows that $\Gamma(-; f^{-1}(-))$ is a sheaf of spaces in the $\infty$-categorical sense. The analogous argument applies to compactly supported sections, see \cite[Remark~5.5.6.4]{HA}. Assuming further that $f$ is an ex-fibration, we can define a sheaf of sequential spectra 
\[
    B \overset{\text{open}}{\supset} U \longmapsto \Gamma(U; \Sigma^k_Bf^{-1}(U)),
\]
where $\Sigma^k_B(-)$ is the fibrewise $k$-fold suspension. Arguing as above, this construction yields a sheaf of spectra. Lastly, when $B$ is locally of singular shape we shall sometimes conflate locally constant sheaves and parameterised spectra. A reference for this equivalence can be found in \cite[Theorem~A.4.19]{HA} for sheaves of spaces, and the spectrum valued case follows from the equivalence $\mathsf{Shv}(B;\catSp) \simeq \mathsf{Shv}(B;\catS) \otimes \catSp$.

\subsubsection{Goodwillie calculus}
Let us now recall some basic facts about Goodwillie calculus. We refer the unfamiliar reader to \cite[\S~1 and 2]{GoodwillieHandbook} for an introduction. We shall be concerned with finitary (i.e. filtered colimit preserving) functors $F \colon \catS_{B//B} \to \catSp$. The calculus provides a tower of functors $\catS_{B//B} \to \catSp$
\[
    F \lra \dotsc \lra P_nF \lra P_{n-1}F \lra \dots \lra P_1F \lra P_0F,
\]
such that $P_0F$ is constant and each $P_nF$ is $n$-excisive \cite[Definition~1.1.2]{GoodwillieHandbook}. The layers 
\[
    D_nF = \fib(P_nF \lra P_{n-1}F)
\]
are $n$-homogeneous \cite[Definition~1.2.1]{GoodwillieHandbook} and as such admit a very explicit formula
\[
    D_nF(Y) \simeq \colim_{B^n \sslash \Sigma_n}\left(\partial_nF \otimes (\Sigma^\infty_BY)^{\boxtimes n}\right)
\]
where $\colim_{B^n \sslash \Sigma_n}(-) \simeq r_!(-)$ for $r \colon B^n \sslash \Sigma_n \to *$, and $\partial_nF \in \catSp^{B^n \sslash \Sigma_n}$ is a parameterised spectrum over $B^n$ with $\Sigma_n$ action called the $n$th derivative of $F$ \cite[Theorem~1.2.3]{GoodwillieHandbook}. If $A \in \catS_{B//B}$, the Goodwillie tower of the functor
\[
    F_A \colon \catS_{A//A} \overset{\text{forget}}{\lra} \catS_{B//B} \overset{F}{\lra} \catSp
\]
is called the tower of $F$ based at $A$ \cite[Definition~1.1.15]{GoodwillieHandbook}. In particular $P_0F_A$ is constant equal to  $F(A)$.

\subsection{Construction}

Let $X$ be a finite dimensional CW-complex. Let $E \in \catS_{X//X}$ be a retractive space over $X$. By composing with the projection and section of $E \to X$, any retractive space $Y \in \catS_{E//E}$ over $E$ can be seen as a retractive space over $X$. We thus write $\Gamma(X;Y)$ for the section space of $Y \to X$, naturally pointed at the given section $X \to E \to Y$. The goal for the remainder of this section is to give a construction of the Goodwillie tower of the functor
\begin{align*}
    \catS_{E//E} &\lra \catSp \\
     Y &\longmapsto \Sigma^\infty \Gamma(X; Y).
\end{align*}
In other words, we will construct the tower of the section space functor $\Sigma^\infty\Gamma(X;-) \colon \catS_{X//X} \to \catSp$ based at $E$. Our strategy is to closely follow Arone's work on mapping spaces \cite{Greg} to guess what the tower should be, and then use the contravariant functoriality on $X$ to check inductively on $\dim X$ that our construction is correct. This latter part is very much inspired by Malkiewich's work \cite{Malkiewich}.

\bigskip

Let us begin with some notation. Let $p \colon E \to X$ and $\pi \colon Y \to E$ be ex-fibrations \cite[Definition~8.1.1]{MaySigurdsson}. Let $\Gamma(E) = \Gamma(X;E)$ be the space of global sections of $p$. We write
\[
    Y^{\ovwedge i} \lra E^i
\]
for the $i$-fold external smash product: if $Y_{e_j}$ denotes the fibre of $Y \to E$ above $e_j$, then $Y_{e_1} \wedge \cdots \wedge Y_{e_i}$ is the fibre of $Y^{\ovwedge i} \to E^i$ above $(e_1,\dotsc,e_i)$. It is also an ex-fibration by \cite[Proposition~8.2.3]{MaySigurdsson}. Let $f \colon B \to X$ be a map from a topological space $B$. We regard $B^i \times \Gamma(E)$ as a space over $f^*E^i$ via the evaluation map 
\begin{align*}
    B^i \times \Gamma(E) &\lra f^*E^i \\
    (b_1,\dotsc,b_i,s) &\longmapsto (f^*s(b_1),\dotsc,f^*s(b_i))
\end{align*}
(beware this is not a fibration in general, in fact it is not even surjective for $i > 1$). Given a retractive space $Z \to f^*E^i$, we write $Z /^{B^i} f^*E^i$ for the retractive space over $B^i$ obtained by collapsing the section of $Z$ fibrewise over $B^i$. Finally, we write $\Sigma^k_{B^i}(-)$ for the $k$-fold fibrewise suspension over $B^i$.

\begin{definition}\label{definition:An}
Let $\widetilde{A}_n(B,Y)$ be the sequential spectrum whose $k$th space is the space of tuples $(s_1,\dotsc,s_n)$ where $s_i$ is a section of
\[
    \cE_i^k(B,Y) := \Sigma^k_{B^i} \big( (f^*Y)^{\ovwedge i} \wedge_{f^*E^i} ((B^i \times \Gamma(E)) \sqcup f^*E^i) \big) /^{B^i} f^*E^i \lra B^i,
\]
and such that the squares
\[
\begin{tikzcd}
\cE_i^k(B,Y) \arrow[r] & \cE_j^k(B,Y) \\
B^i \arrow[u, "s_i"] \arrow[r]                                                       & B^j \arrow[u, "s_j"']                         
\end{tikzcd}
\]
commute for every surjection $\{1,\dotsc,j\} \twoheadrightarrow \{1,\dotsc,i\}$.
\end{definition}

Our construction is chosen so as to generalise Arone's version for mapping spaces \cite[Definition~7]{Greg}. We observe that we have a map
\[
    \Gamma(B;f^*Y) \times_{\Gamma(B;f^*E)} \Gamma(E) \lra \Gamma(B^i; \cE^0_i(B,Y))
\]
sending a pair of sections $(s,g)$ to the section
\begin{equation}\label{eqn:adjointmap}
    (b_1,\dotsc,b_i) \longmapsto s(b_1) \wedge \dots \wedge s(b_i) \wedge ((b_1,\dotsc,b_i), g).
\end{equation}
By composing with the unit map $\cE^0_i(B,Y) \to \Omega^k_{B^i} \Sigma^k_{B^i}\cE^0_i(B,Y) = \Omega^k_{B^i}\cE^k_i(B,Y)$, we obtain a map
\[
    \Gamma(B;f^*Y) \times_{\Gamma(B;f^*E)} \Gamma(E) \lra \Gamma(B^i; \Omega^k_{B^i}\cE^k_i(B,Y)) \simeq \Omega^k \Gamma(B^i; \cE^k_i(B,Y)).
\]
By adjunction, we thus obtain a map of sequential spectra
\[
    \Sigma^\infty\left(\Gamma(B;f^*Y) \times_{\Gamma(B;f^*E)} \Gamma(E)\right) \lra \widetilde{A}_n(B,Y).
\]
Projecting onto the second factor also yields a map
\[
    \Sigma^\infty\left(\Gamma(B;f^*Y) \times_{\Gamma(B;f^*E)} \Gamma(E)\right) \to \Sigma^\infty\Gamma(E),
\]
which we may add to the one above to obtain
\begin{equation}\label{eqn:goodwillie-tower-map}
    \Sigma^\infty\left(\Gamma(B;f^*Y) \times_{\Gamma(B;f^*E)} \Gamma(E)\right) \lra A_n(B,Y) := \Sigma^\infty\Gamma(E) \oplus \widetilde{A}_n(B,Y).
\end{equation}
Before embarking on the analysis of this map, and showing that it is the universal $n$-polynomial approximation, it will be convenient to verify that our point-set constructions have good homotopical properties. We do this in the lemmas below.

\begin{lemma}\label{lemma:exfibration}
Let $R \to S$ and $S \to T$ be ex-fibrations. Let $U \to S$ be a well-sectioned retractive space such that the composite map $U \to S \to T$ is an ex-fibration. Then
\[
    R \wedge_S U /^T S \lra T
\]
is an ex-fibration.
\end{lemma}
\begin{proof}
We first show that $R \wedge_S U \to T$ is an ex-fibration. First, $R \times_S U \to U$ is an ex-fibration as it is the pullback of the ex-fibration $R \to S$ along $U \to S$ \cite[Proposition~8.2.2]{MaySigurdsson}. Hence, by composition, $R \times_S U \to U \to T$ is also an ex-fibration. The smash product $R \wedge_S U$ is the pushout of $S \leftarrow R \vee_S U \to R \times_S U$ hence is an ex-fibration over $T$ by \cite[Proposition~8.2.1.(ii)]{MaySigurdsson}. Finally, the fibrewise quotient is the pushout of $T \leftarrow S \to R \wedge_S U$ which is seen to be an ex-fibration over $T$ by the same argument.
\end{proof}

\begin{lemma}[Compare {\cite[Proposition~3.13]{Malkiewich}}]\label{lemma:invariance-in-B}
For $Y \to E$ an ex-fibration, the functor $\widetilde{A}_n(-,Y)$ is a homotopy functor on the category of Quillen cofibrant topological spaces over $X$. That is, it preserves weak homotopy equivalences between cofibrant topological spaces over $X$.
\end{lemma}
\begin{proof}
Recall that in the slice Quillen model structure, a map $B \to B'$ of topological spaces over $X$ is an acyclic cofibration if its underlying map is. By Ken Brown's lemma, it suffices to show that such a map is sent to a levelwise equivalence $\widetilde{A}_n(B',Y) \to \widetilde{A}_n(B,Y)$. Although the evaluation map $B^i \times \Gamma(E) \to f^*E^i$ is not a fibration, the composition $B^i \times \Gamma(E) \to f^*E^i \to B^i$ is simply the projection hence trivially a fibration. We may therefore apply \cref{lemma:exfibration}, and the fact that fibrewise suspension and $f^*$ preserve ex-fibrations \cite[Propositions~8.2.1 and 8.2.2]{MaySigurdsson}, to conclude that the map
\[
    \cE_i^k(B,Y) = \Sigma^k_{B^i} \big( (f^*Y)^{\ovwedge i} \wedge_{f^*E^i} ((B^i \times \Gamma(E)) \sqcup f^*E^i) \big) /^{B^i} f^*E^i \lra B^i
\]
is an ex-fibration hence a Serre fibration. We are reduced to showing that for section spaces of Serre fibrations, restricting along an acyclic cofibration is a weak equivalence. This follows from standard lifting arguments, see e.g. the proof of \cite[Proposition~3.13]{Malkiewich}.
\end{proof}

\begin{lemma}\label{lemma:invariance-in-Y}
Let $B$ be a CW-complex. The functor $\widetilde{A}_n(B,-)$ sends homotopy equivalences between ex-fibrations to homotopy equivalences.
\end{lemma}
\begin{proof}
Any homotopy equivalence $Y \to Y'$ induces a homotopy equivalence $\cE_i^k(B,Y) \to \cE_i^k(B,Y')$. As the base $B$ is a CW-complex, this induces an equivalence between the section spaces.
\end{proof}

We can now observe that our construction $A_n(-,-)$ is indeed a homotopy functor. The slice Quillen model category on $\catTop_{/X}$ presents the slice $\infty$-category $\catS_{/X}$, see e.g. \cite[Corollary~7.6.13]{Cisinski2019}. Likewise, the Hurewicz model structure on the category of retractive topological spaces $\catTop_{E//E}$ presents the slice $\infty$-category $\catS_{E//E}$, hence so does the subcategory of ex-fibrations by the existence of an ex-fibrant replacement functor \cite[Section~8.3]{MaySigurdsson}. By \cref{lemma:invariance-in-B,lemma:invariance-in-Y} we thus obtain a functor
\[
    A_n(-,-) \colon \catS_{/X}^\opp \times \catS_{E//E} \lra \catSp.
\]

For later manipulations, it will also be useful to know that we may replace the evaluation map by a fibration without changing the homotopy type of our construction:
\begin{lemma}\label{lemma:ev-fibrant-replacement}
Let $B^i \times \Gamma(E) \sqcup f^*E^i \overset{\simeq}{\lra} W$ be a homotopy equivalence of retractive spaces over $f^*E^i$ such that $W \to f^*E^i$ is a fibration. Then
\[
    \Sigma^k_{B^i} \big( (f^*Y)^{\ovwedge i} \wedge_{f^*E^i} ((B^i \times \Gamma(E)) \sqcup f^*E^i) \big) /^{B^i} f^*E^i \lra \Sigma^k_{B^i} \big( (f^*Y)^{\ovwedge i} \wedge_{f^*E^i} W \big) /^{B^i} f^*E^i
\]
is a homotopy equivalence of ex-fibrations over $B^i$.
\end{lemma}
\begin{proof}
By \cite[Proposition~8.2.6]{MaySigurdsson}, the map obtained by smashing
\[
    (f^*Y)^{\ovwedge i} \wedge_{f^*E^i} ((B^i \times \Gamma(E)) \sqcup f^*E^i) \lra (f^*Y)^{\ovwedge i} \wedge_{f^*E^i} W
\]
is a homotopy equivalence between ex-fibrations over $B^i$. The suspension and the quotient are defined by homotopy pushouts which preserve homotopy equivalences.
\end{proof}

\subsection{Polynomiality}\label{subsection:polynomiality}

We now show that our bivariate functor $A_n(-,-)$ is $n$-excisive in each variable.
\begin{proposition}\label{proposition:first-variable-excision}
For any $Y \in \catS_{E//E}$, the functor $A_n(-,Y)$ sends strongly cocartesian $(n+1)$-cubes to cartesian cubes.
\end{proposition}
\begin{proof}
This is analogous to \cite[Proposition~3.15]{Malkiewich}.
\end{proof}

To show polynomiality in the second variable, we proceed by induction on $n$. Let
\[
    D_n(B,Y) = \fib\left( A_n(B,Y) \lra A_{n-1}(B,Y) \right)
\]
be the fibre of the natural map which forgets the $n$th section $s_n$.

\begin{lemma}\label{lemma:layer-computation}
Let $B$ be a finite CW-complex. Denote by $\cE^\infty_n(B,Y) = \Sigma^\infty_{B^n} \cE^0_n(B,Y)$ the fibrewise suspension spectrum. There is a natural equivalence of spectra
\[
    D_n(B,Y) \simeq \Gamma_c(\UConf_n(B); \cE^\infty_n(B,Y)_{\Sigma_n}) \simeq \Gamma_c(\Conf_n(B); \cE^\infty_n(B,Y))_{h\Sigma_n},
\]
where $\Gamma_c(-)$ denotes the spectrum of compactly supported sections, $\cE^\infty_n(B,Y)_{\Sigma_n}$ is the quotient bundle over the unordered configuration space $\UConf_n(B)$, and $\Sigma_n$ acts naturally by permutations.
\end{lemma}
\begin{proof}
We first show that the projection map $\widetilde{A}_n(B,Y) \to \widetilde{A}_{n-1}(B,Y)$ is a levelwise fibration of sequential spectra, and then compute the strict fibre. Let $\widetilde{A}^k_n(B,Y)$ and $\widetilde{A}^k_{n-1}(B,Y)$ be the $k$th spaces of the respective spectra. Let $\Delta \subset B^n$ be the fat diagonal. The compatibility conditions on $(s_1,\dotsc,s_{n-1}) \in \widetilde{A}^k_{n-1}(B,Y)$ show that they glue to a section of the restriction $\cE^k_n(B,Y)|_\Delta \to \Delta$. Furthermore, the tuple extends to an element $(s_1,\dotsc,s_{n-1},s_n) \in \widetilde{A}^k_n(B,Y)$ if and only if there exists a $\Sigma_n$-equivariant section $s_n$ of $\cE^k_n(B,Y)$ which extends the given section over $\Delta$. In other words, we have a strict pullback of spaces:
\[
\begin{tikzcd}
|[label={[label distance=-2mm]-45:\lrcorner}]| {\widetilde{A}^k_n(B,Y)} \arrow[d] \arrow[r] & \Gamma(\cE^k_n(B,Y))^{\Sigma_n} \arrow[d, "\mathrm{res}"] \\
{\widetilde{A}^k_{n-1}(B,Y)} \arrow[r]       & \Gamma(\cE^k_n(B,Y)|_\Delta)^{\Sigma_n}  
\end{tikzcd}
\]
The symmetric group $\Sigma_n$ acts on $\Delta$, $B^n$, and $\cE^k_n(B,Y)$ by permutations, hence $\cE^k_n(B,Y) \to B^n$ is a $\Sigma_n$-fibration. As $\Delta \subset B^n$ is a cofibration, the restriction map $\mathrm{res}$ is a fibration, hence so is the pulled back map. As $B$ is a CW-complex, hence cofibrant, our point-set formula for $A_n(B,Y)$ has the correct homotopy type. Hence we readily compute the homotopy fibre from the diagram
\begin{align*}
    \fib\big(\widetilde{A}_n(B,Y) \to \widetilde{A}_{n-1}(B,Y)\big) &\simeq \fib\big(\Gamma(\cE^\infty_n(B,Y))^{\Sigma_n} \to \Gamma(\cE^\infty_n(B,Y)|_\Delta)^{\Sigma_n}\big) \\
        &\simeq \fib\big(\Gamma(\cE^\infty_n(B,Y)) \to \Gamma(\cE^\infty_n(B,Y)|_\Delta)\big)^{\Sigma_n}
\end{align*}
In this formula, the symbol $(-)^{\Sigma_n}$ denotes the fixed point functor applied levelwise to a sequential spectrum (i.e. the categorical fixed points). In particular, it commutes with taking the fibre as the restriction map is a levelwise fibration. Furthermore this fibre is the homotopy fibre, hence it agrees with the spectrum of compactly supported sections by the localisation sequence of sheaves of spectra \cite[Corollary~4.5]{Volpe}:
\[
    \fib\big(\Gamma(\cE^\infty_n(B,Y)) \to \Gamma(\cE^\infty_n(B,Y)|_\Delta)\big)^{\Sigma_n} \simeq \Gamma_c\left(B^n - \Delta; \cE^\infty_n(B,Y)\right)^{\Sigma_n}.
\]
As the symmetric group action on the configuration space $B^n - \Delta$ is free, an equivariant section over the ordered configuration space $B^n - \Delta$ is the same as a section of the quotient bundle over the unordered configuration space $\UConf_n(B) = (B^n - \Delta)/\Sigma_n$, that is
\[
    \Gamma_c\left(B^n - \Delta; \cE^\infty_n(B,Y)\right)^{\Sigma_n} \simeq  \Gamma_c(\UConf_n(B); \cE^\infty_n(B,Y)_{\Sigma_n}).
\]
By the Adams isomorphism \cite[Theorem~II.7.1]{MayLewis}, it is also equivalent to $\Gamma_c(\Conf_n(B); \cE^\infty_n(B,Y))_{h\Sigma_n}$.
\end{proof}

\begin{corollary}\label{corollary:homogeneity}
Let $B$ be a finite CW-complex with a map $f \colon B \to X$. Then the functor
\[
    D_n(B,-) \colon \catS_{E//E} \to \catSp
\]
is $n$-homogeneous in the sense of Goodwillie. Furthermore, if $B$ is a closed manifold then there is a natural equivalence
\[
    D_n(B,Y) \simeq \colim_{f^*E^n \sslash \Sigma_n}\left( f^*(\Sigma_E^\infty Y)^{\boxtimes n} \otimes \mathrm{ev}_!\bS^0 \otimes \overline{p}^*\bS^{-TB^n} \otimes \overline{p}^*\omega_!\bS^0\right)
\]
where $\overline{p} \colon f^*E^n \sslash \Sigma_n \to B^n \sslash \Sigma_n$ is the projection, $\mathrm{ev} \colon B^n \sslash \Sigma_n \times \Gamma(E) \to f^*E^n \sslash \Sigma_n$ is the evaluation map, and $\omega \colon \UConf_n(B) \to B^n \sslash \Sigma_n$ is the inclusion. In particular, when $X$ is a closed manifold, the parameterised spectrum
\[
    \mathrm{ev}_!\bS^0 \otimes p^*\bS^{-TX^n} \otimes p^*\omega_!\bS^0 \in \catSp^{E^n \sslash \Sigma_n}
\]
is the derivative of the functor $D_n(X,-) \colon \catS_{E//E} \to \catSp$.
\end{corollary}
\begin{proof}
We start by assuming that $B$ is a closed manifold. Let $q \colon \UConf_n(B) \to *$ be the projection map. We will work with spectrum valued sheaves and in particular use the exceptional pushforward and pullback functors $q^{\sh}_!$ and $q^{\sh,!}$, see \cite[\S~6]{Volpe} for a construction. In this setting, we may rewrite the conclusion of \cref{lemma:layer-computation} as
\begin{equation}\label{eqn:derivative-with-sheaves}
    D_n(B,Y) \simeq q^{\sh}_! ((\cE^\infty_n(B,Y)|_{\UConf_n(B)})_{\Sigma_n}).
\end{equation}
As $\UConf_n(B)$ is a manifold, the dualizing sheaf $q^{\sh,!}(\bS^0) \simeq \bS^{T\UConf_n(B)}$ is invertible and agrees with the stable spherical fibration associated to its tangent bundle \cite[Theorem~7.14]{Volpe}. Taking the left adjoints in the formula $q^{\sh,!}(-) \simeq q^{\sh,*}(-) \otimes q^{\sh,!}(\bS^0)$ \cite[Proposition~6.18]{Volpe}, we obtain
\[
    q_\#^\sh(- \otimes q^{\sh,!}(\bS^0)^{-1}) \simeq q^\sh_!(-),
\]
where $q_\#^\sh$ is the left adjoint to $q^{\sh,*}$, which exists because $\UConf_n(B)$ is locally of constant shape, see \cite[Corollary~3.21]{Volpe} and use that $\UConf_n(B)$ is open in the CW-complex $B^n/\Sigma_n$ (which is always hypercomplete \cite{MOHoyois}). As the pullback functor $q^{\sh,*}$ preserves local systems, i.e. parameterised spectra, it agrees with $q^*$. Hence $q_\#^\sh$ agrees with $q_!$ when evaluated on local systems. Therefore
\[
    D_n(B,Y) \simeq q^{\sh}_! ((\cE^\infty_n(B,Y)|_{\UConf_n(B)})_{\Sigma_n}) \simeq q_! \left((\cE^\infty_n(B,Y)|_{\UConf_n(B)})_{\Sigma_n} \otimes \bS^{-T\UConf_n(B)}\right).
\]
Before going further, it is useful to notice that
\[
    (\cE^\infty_n(B,Y)|_{\UConf_n(B)})_{\Sigma_n} \simeq \omega^*\cE^\infty_n(B,Y) \quad \text{and} \quad \bS^{-T\UConf_n(B)} \simeq \omega^*\bS^{-TB^n}.
\]
We will now manipulate this formula to observe that it is indeed a homogeneous functor in $Y$. From the definition of $\cE^\infty_n(B,Y)$ and using $\mathrm{ev}_!(S^0)$ as a fibrant replacement together with \cref{lemma:ev-fibrant-replacement}, we have an equivalence
\[
    \cE^\infty_n(B,Y) \simeq \overline{p}_!\left(f^*(\Sigma_E^\infty Y)^{\boxtimes n} \otimes \mathrm{ev}_!\bS^0 \right)
\]
of parameterised spectra. Hence
\begin{align*}
    D_n(B,Y) &\simeq q_! \omega^*\left( \overline{p}_!\left(f^*(\Sigma_E^\infty Y)^{\boxtimes n} \otimes \mathrm{ev}_!\bS^0 \right) \otimes \bS^{-TB^n} \right) \\
        &\simeq \overline{q}_! \omega_! \omega^*\left( \overline{p}_!\left(f^*(\Sigma_E^\infty Y)^{\boxtimes n} \otimes \mathrm{ev}_!\bS^0 \right) \otimes \bS^{-TB^n} \right) \\
        &\simeq \overline{q}_! \left(\overline{p}_!\left(f^*(\Sigma_E^\infty Y)^{\boxtimes n} \otimes \mathrm{ev}_!\bS^0 \right) \otimes \bS^{-TB^n} \otimes \omega_!\bS^0 \right) \\
        &\simeq \overline{q}_!\overline{p}_! \left( f^*(\Sigma_E^\infty Y)^{\boxtimes n} \otimes \mathrm{ev}_!\bS^0 \otimes \overline{p}^*\bS^{-TB^n} \otimes \overline{p}^*\omega_!\bS^0\right) \\
        &\simeq \colim_{f^*E^n \sslash \Sigma_n}\left( f^*(\Sigma_E^\infty Y)^{\boxtimes n} \otimes \mathrm{ev}_!\bS^0 \otimes \overline{p}^*\bS^{-TB^n} \otimes \overline{p}^*\omega_!\bS^0\right)
\end{align*}
where we have written $\overline{q} \colon B^n \sslash \Sigma_n \to *$ for the unique map, used $q = \overline{q} \circ \omega$ for the second equivalence, the projection formula for the third and fourth, and rewritten the last pushforward using homotopy orbits. As $f^* \colon \catSp^E \to \catSp^{f^*E}$ is linear, $D_n(B,Y)$ is visibly $n$-homogeneous.

Let us now treat the general case of a finite CW-complex $B$ by induction on cell attachments. More precisely, up to homotopy and using \cref{lemma:invariance-in-B} we may assume that $B^n$ is obtained from the diagonal $\Delta \subset B^n$ by gluing free cells of the form $\Sigma_n \times D^d$. We write
\[
    \Delta = L_0 \subset L_1 \subset L_2 \subset \dots \subset L_N = B^n
\]
for a sequence of such cell attachments where $L_i$ is obtained from $L_{i-1}$ by attaching a free cell $\Sigma_n \times D^{d_i}$. Then, the fibre of the restriction map
\begin{align*}
    \fib\left( \Gamma(L_i; f^*\cE^\infty_n(X,Y))^{\Sigma_n} \to \Gamma(L_{i-1}; f^*\cE^\infty_n(X,Y))^{\Sigma_n} \right) &\simeq \Gamma(D^{d_i}, S^{d_i-1}; f^*\cE^\infty_n(X,Y)) \\
    &\simeq \Sigma^{-d_i} \cE^\infty_n(X,Y)|_{f(p_i)}
\end{align*}
where $p_i \in L_i$ is a chosen centre in one of the discs $D^{d_i}$. This is an $n$-homogeneous functor by what we have established in the $0$-dimensional case above. By induction on $i$, it follows that
\[
    \fib\left( \Gamma(B^n; f^*\cE^\infty_n(X,Y))^{\Sigma_n} \to \Gamma(\Delta; f^*\cE^\infty_n(X,Y))^{\Sigma_n} \right)
\]
is an $n$-homogeneous functor of $Y$. The result then follows by \cref{lemma:layer-computation}.
\end{proof}

\begin{proposition}\label{proposition:second-variable-excision}
Let $B$ be a finite CW-complex. Then the functor $A_n(B,-) \colon \catS_{E//E} \to \catSp$ is $n$-polynomial in the sense of Goodwillie.
\end{proposition}
\begin{proof}
This follows directly by induction on $n$ using the fibre sequences
\[
    D_n(B,-) \lra A_n(B,-) \lra A_{n-1}(B,-)
\]
and \cref{corollary:homogeneity}.
\end{proof}

\subsection{Agreement on points}

Ultimately, our goal is to show that the natural map~\eqref{eqn:goodwillie-tower-map}
\[
    \Sigma^\infty\left(\Gamma(B;f^*Y) \times_{\Gamma(B;f^*E)} \Gamma(E)\right) \lra A_n(B,Y)
\]
is the universal $n$-polynomial approximation in the sense of Goodwillie. Our strategy is to use the functoriality in the variable $B$ and induction on its dimension. In this section, we treat the case $\dim B = 0$. 

\begin{lemma}\label{lemma:splitting-An}
Let $B = \{b_1,\dotsc,b_l\}$ be a finite $0$-dimensional CW-complex with a map $f \colon B \to X$, and write $x_i = f(b_i)$. There is a natural equivalence
\[
    \widetilde{A}_n(B,Y) \simeq \bigoplus_{\substack{\varnothing \neq S \subset \{1,\dotsc,l\}\\ |S| \leq n}} \Sigma^\infty\left( \overline{\bigwedge}_{i \in S} Y_{x_i} \wedge_{\prod_{i \in S} E_{x_i}} \left(\Gamma(E) \sqcup \prod_{i \in S} E_{x_i}\right) \right) \bigg/ \prod_{i \in S} E_{x_i}.
\]
\end{lemma}
\begin{proof}
A tuple of sections $(s_1,\dotsc,s_n)$ in the $k$th space of $\widetilde{A}_n(B,Y)$ is determined by the values taken by $s_n$, i.e. it is the data of a $|B^n| = l^n$-tuple of points in various fibres of $\cE^k_n(B,Y)$ subject to compatibility conditions. These conditions precisely impose that the value at a tuple $(t_1,\dotsc,t_n) \in B^n$ is determined by the underlying set $S = \{t_1,\dotsc,t_n\}$, which we identify with a non-empty subset of $\{1,\dotsc,l\}$ of size at most $n$.
\end{proof}

\begin{lemma}\label{lemma:splitting-suspension-product}
Let $x_1,\dotsc,x_l \in X$. There is a natural equivalence of parameterised spectra over $\prod_{i=1}^l E_{x_i}$
\[
    \Sigma^\infty \prod_{i = 1}^l Y_{x_i} \overset{\simeq}{\lra} \bigoplus_{\varnothing \neq S \subset \{1,\dotsc,l\}} \mathrm{pr}_S^* \Sigma^\infty \overline{\bigwedge}_{i \in S} Y_{x_i}
\]
where $\mathrm{pr}_S \colon \prod_{i=1}^l E_{x_i} \to \prod_{i \in S} E_{x_i}$ is the projection, and $\Sigma^\infty$ denotes the fibrewise suspension.
\end{lemma}
\begin{proof}
For each $S$ we obtain a map $\prod_{i=1}^l Y_{x_i} \to \mathrm{pr}_S^* \overline{\bigwedge}_{i \in S} Y_{x_i}$ of retractive spaces over $\prod_{i=1}^l E_{x_i}$ by projecting from the cartesian product to the smash product. After taking suspension spectra, these assemble into a map into the direct sum:
\[
    \Sigma^\infty \prod_{i = 1}^l Y_{x_i} \lra \bigoplus_{\varnothing \neq S \subset \{1,\dotsc,l\}} \mathrm{pr}_S^* \Sigma^\infty \overline{\bigwedge}_{i \in S} Y_{x_i}
\]
which we can check is an equivalence fibrewise: it reduces to the standard splitting of the suspension of a product.
\end{proof}

We write $\conn(A)$ for the connectivity of a space $A$, and $\conn(A/B)$ for the connectivity of a map $A \to B$ (which is also equal to $\conn(\fib(A \to B))+1$).

\begin{proposition}\label{proposition:zero-dimensional-connectivity}
Let $B$ be a finite $0$-dimensional CW-complex with a map $f \colon B \to X$. Then the map~\eqref{eqn:goodwillie-tower-map}
\[
    \Sigma^\infty\big( \Gamma(B; f^*Y) \times_{\Gamma(B; f^*E)} \Gamma(E) \big)  \lra A_n(B,Y)
\]
is $(n+1)\conn(Y/E)$-connected.
\end{proposition}
\begin{proof}
Let us write $B = \{b_1,\dotsc,b_l\}$ and $x_i = f(b_i)$, so that $\Gamma(B; f^*Y) = \prod_{i = 1}^l Y_{x_i}$ and $\Gamma(B; f^*E) = \prod_{i = 1}^l E_{x_i}$. It will be clearer to distinguish between a space in a slice category and its associated unstraightened functor. So let us write $\cY_i \colon E_{x_i} \to \catS_*$ for the functor associated to the ex-fibration $Y_{x_i} \to E_{x_i}$, so that the external product $\prod_{i=1}^l \cY_i$ corresponds to the bundle $\prod_{i = 1}^l Y_{x_i} \to \prod_{i = 1}^l E_{x_i}$. By definition,
\[
    \Gamma(B; f^*Y) \times_{\Gamma(B; f^*E)} \Gamma(E) = \mathrm{ev}^*(\Gamma(B; f^*Y))
\]
where $\mathrm{ev} \colon \Gamma(E) \to \Gamma(B; f^*E)$ is the restriction map along $f$. Hence $\mathrm{ev}^*\prod_{i=1}^l \cY_i$ corresponds to the bundle $\Gamma(B; f^*Y) \times_{\Gamma(B; f^*E)} \Gamma(E) \to \Gamma(E)$. Expressing that the total space of a fibration is the colimit (in the category of unpointed spaces) of its associated functor, we obtain an equivalence
\[
    \prod_{i = 1}^l Y_{x_i} \times_{\prod_{i = 1}^l E_{x_i}} \Gamma(E) \simeq \colim_{\Gamma(E)} \mathrm{ev}^*\prod_{i=1}^l \cY_i.
\]
Upon taking suspension spectra, which commutes with the colimit, we get
\[
    \Sigma^\infty_+\left(\prod_{i = 1}^l Y_{x_i} \times_{\prod_{i = 1}^l E_{x_i}} \Gamma(E)\right) \simeq \colim_{\Gamma(E)} \mathrm{ev}^*\Sigma^\infty_+\left(\prod_{i=1}^l \cY_i\right) \simeq \bigoplus_{S \subset \{1,\dotsc,l\}} \colim_{\Gamma(E)} \mathrm{ev}^*\mathrm{pr}_S^* \Sigma^\infty \overline{\bigwedge}_{i \in S} \cY_i
\]
where the second equivalence follows from \cref{lemma:splitting-suspension-product}. Be aware that we have to add a basepoint before taking the suspension spectrum, hence $S = \varnothing$ is also allowed in the sum. The projection formula yields
\[
    \colim_{\Gamma(E)} \mathrm{ev}^*\mathrm{pr}_S^* \Sigma^\infty \overline{\bigwedge}_{i \in S} \cY_i \simeq \colim_{\prod_{i \in S} E_{x_i}} \Sigma^\infty \overline{\bigwedge}_{i \in S} \cY_i \otimes (\mathrm{pr}_S \circ \mathrm{ev})_!\bS^0.
\]
This is a colimit of parameterised suspension spectra, so we may compute it first in parameterised pointed spaces and then suspend. Noticing that the total space of an ex-fibration modulo its section is the colimit of its associated parameterised pointed space, we get
\begin{align*}
    \colim_{\prod_{i \in S} E_{x_i}} \Sigma^\infty \overline{\bigwedge}_{i \in S} \cY_i \otimes (\mathrm{pr}_S \circ \mathrm{ev})_!\bS^0 &\simeq \Sigma^\infty \left(\overline{\bigwedge}_{i \in S} Y_{x_i} \wedge_{\prod_{i \in S} E_{x_i}} (\mathrm{pr}_S \circ \mathrm{ev})_!S^0 \right)/ \prod_{i \in S} E_{x_i} \\
    &= \Sigma^\infty \left(\overline{\bigwedge}_{i \in S} Y_{x_i} \wedge_{E_S} (\mathrm{pr}_S \circ \mathrm{ev})_!S^0 \right)/ E_S,
\end{align*}
where we write $E_S = \prod_{i \in S} E_{x_i}$ to lighten the formulas.
Finally, by definition of the $(-)_!$ functor, $(\mathrm{pr}_S \circ \mathrm{ev})_!S^0$ can be computed by taking a fibrant replacement of the retractive space
\[
    B^i \times \Gamma(E) \sqcup E_S \to E_S.
\]
Hence, by \cref{lemma:ev-fibrant-replacement}, we obtain
\[
    \left(\overline{\bigwedge}_{i \in S} Y_{x_i} \wedge_{E_S} (\mathrm{pr}_S \circ \mathrm{ev})_!S^0 \right)/ E_S \simeq \left(\overline{\bigwedge}_{i \in S} Y_{x_i} \wedge_{E_S} (\Gamma(E) \sqcup E_S)\right)/ E_S.
\]
Summing up, we have obtained an equivalence
\[
    \Sigma^\infty\big( \Gamma(B; f^*Y) \times_{\Gamma(B; f^*E)} \Gamma(E) \big) \simeq \Sigma^\infty\Gamma(E) \oplus \bigoplus_{\varnothing \neq S \subset \{1,\dotsc,l\}} \Sigma^\infty \left(\overline{\bigwedge}_{i \in S} Y_{x_i} \wedge_{E_S} (\Gamma(E) \sqcup E_S)\right)/ E_S.
\]
Together with \cref{lemma:splitting-An}, we see that our map~\eqref{eqn:goodwillie-tower-map}
\[
    \Sigma^\infty\big( \Gamma(B; f^*Y) \times_{\Gamma(B; f^*E)} \Gamma(E) \big) \lra A_n(B,Y)
\]
can be identified with the projection onto the summands with $|S| \leq n$, hence its fibre is
\[
    \bigoplus_{S \subset \{1,\dotsc,l\}, \ |S| > n} \Sigma^\infty \left(\overline{\bigwedge}_{i \in S} Y_{x_i} \wedge_{E_S} (\Gamma(E) \sqcup E_S)\right)/ E_S.
\]
The summand corresponding to $S$ is $(|S|\conn(Y/E)-1)$-connected, hence the direct sum is itself $((n+1)\conn(Y/E)-1)$-connected.
\end{proof}

\subsection{Cell attachments}

We will need the following crude connectivity estimate.
\begin{lemma}\label{lemma:crude-connectivity}
Let $A \to B$ be a map between finite CW-complexes of dimension $\leq d$. Then the restriction map
\[
    \Gamma(Y|_{B}) \times_{\Gamma(E|_{B})} \Gamma(E) \lra \Gamma(Y|_{A}) \times_{\Gamma(E|_{A})} \Gamma(E)
\]
is at least $(\conn(Y/E)-d-1)$-connected.
\end{lemma}
\begin{proof}
We will use the observation that the pullback of an $i$-connected map is $i$-connected. For the duration of this proof, we write $(-)|_B$, resp. $(-)|_A$, for the pullback to $B$, resp. $A$. Consider the diagram of pullback squares:
\[
\begin{tikzcd}
|[label={[label distance=-2mm]-45:\lrcorner}]|\Gamma(Y|_{B}) \times_{\Gamma(E|_{B})} \Gamma(E) \arrow[d] \arrow[r] & |[label={[label distance=-2mm]-45:\lrcorner}]|\Gamma(Y|_{A}) \times_{\Gamma(E|_{A})} \Gamma(E) \arrow[d] \arrow[r]      & \Gamma(E) \arrow[d]      \\
\Gamma(Y|_{B}) \arrow[r]    & |[label={[label distance=-2mm]-45:\lrcorner}]|\Gamma(Y|_{A}) \times_{\Gamma(E|_{A})} \Gamma(E|_{B}) \arrow[r] \arrow[d] & \Gamma(E|_{B}) \arrow[d] \\
 & \Gamma(Y|_{A}) \arrow[r]   & \Gamma(E|_{A})          
\end{tikzcd}
\]
As a direct consequence of our assumptions, $\Gamma(Y|_{A}) \to \Gamma(E|_{A})$ is $(\conn(Y/E)-d)$-connected. Hence the pulled back map
\[
    \Gamma(Y|_{A}) \times_{\Gamma(E|_{A})} \Gamma(E|_{B}) \lra \Gamma(E|_{B})
\]
is also $(\conn(Y/E)-d)$-connected. Likewise, the middle horizontal composition $\Gamma(Y|_{B}) \to \Gamma(E|_{B})$ is $(\conn(Y/E)-d)$-connected. A 2-out-of-3 computation then shows that
\[
    \Gamma(Y|_{B}) \lra \Gamma(Y|_{A}) \times_{\Gamma(E|_{A})} \Gamma(E|_{B})
\]
is $(\conn(Y/E)-d-1)$-connected. Therefore the pulled back map
\[
    \Gamma(Y|_{B}) \times_{\Gamma(E|_{B})} \Gamma(E) \lra \Gamma(Y|_{A}) \times_{\Gamma(E|_{A})} \Gamma(E)
\]
is $(\conn(Y/E)-d-1)$-connected.
\end{proof}

\begin{proposition}\label{proposition:order-n-agreement}
Let $B$ be a $d$-dimensional finite CW-complex with a map $f \colon B \to X$. Then the natural map
\[
    \Sigma^\infty\big( \Gamma(B; f^*Y) \times_{\Gamma(B; f^*E)} \Gamma(E) \big)  \lra A_n(B,Y)
\]
is $(n+1)(\conn(Y/E)-d)$-connected. In particular, when $f \colon B = X$ is the identity, the map
\[
    \Sigma^\infty\Gamma(X;Y) \lra A_n(X,Y)
\]
is $(n+1)(\conn(Y/E)-\dim(X))$-connected.
\end{proposition}
\begin{proof}
We will proceed  by induction on the dimension of $B$. The zero-dimensional case was treated in \cref{proposition:zero-dimensional-connectivity}. Suppose then that the conclusion holds for finite CW-complexes of dimension $< d$. Let $B$ be a $d$-dimensional finite CW-complex. By attaching punctured cells\footnote{See for example the ``layer cake space" of \cite[Proof of Proposition~2.3 p.~534]{Malkiewich}. Though beware the little misprint: the formula should also include the entire boundary of the cube.}, $B$ can be written as the pushout of a strongly cocartesian $(n+1)$-cube
\[
    \cP\colon \{1,\dotsc,n+1\} \supset S \longmapsto \cP(S) 
\]
with $\cP(\{1,\dotsc,n+1\}) = B$ and such that $\cP(S)$ is homotopy equivalent to a finite CW-complex of dimension $<d$ for $S \neq \{1,\dotsc,n+1\}$. By \cref{proposition:first-variable-excision}, the cube
\[
    A_n(\cP, Y) \colon \{1,\dotsc,n+1\} \supset S \longmapsto A_n(\cP(S),Y)
\]
is cartesian. On the other hand, the cube
\[
    \Gamma(Y|_{\cP}) \times_{\Gamma(E|_{\cP})} \Gamma(E) \colon \{1,\dotsc,n+1\} \supset S \longmapsto \Gamma(Y|_{\cP(S)}) \times_{\Gamma(E|_{\cP(S)})} \Gamma(E)
\]
is strongly cartesian. The extremal legs of the cube
\[
    \Gamma(Y|_{\cP(\{i\})}) \times_{\Gamma(E|_{\cP(\{i\})})} \Gamma(E) \lra \Gamma(Y|_{\cP(\varnothing)}) \times_{\Gamma(E|_{\cP(\varnothing)})} \Gamma(E)
\]
are $(\conn(Y/E)-d)$-connected by \cref{lemma:crude-connectivity}. Hence the Blakers--Massey theorem \cite[Theorem~2.4]{Goodwillie2} shows that the cube is also $(n+(n+1)(\conn(Y/E)-d))$-cocartesian. Taking suspension spectra, we get a cube which is also $(n+(n+1)(\conn(Y/E)-d))$-cocartesian. By \cite[Remark~1.19]{Goodwillie2}, such a cube of spectra is also $(n+1)(\conn(Y/E)-d)$-cartesian. By induction, for $S \neq \{1,\dotsc,n+1\}$, the map
\[
    \Sigma^\infty\left( \Gamma(Y|_{\cP(S)}) \times_{\Gamma(E|_{\cP(S)})} \Gamma(E) \right) \lra A_n(\cP(S), Y)
\]
is $(n+1)(\conn(Y/E)-d+1)$-connected. Hence the map on the limit
\[
    \Sigma^\infty\left( \Gamma(Y|_{B}) \times_{\Gamma(E|_{B})} \Gamma(E) \right) \lra A_n(B, Y)
\]
has connectivity $(n+1)(\conn(Y/E)-d+1) - (n+1) = (n+1)(\conn(Y/E)-d)$.
\end{proof}

\subsection{Conclusion}

We summarise our results in the following theorem. It should be of general interest for applications outside of the scope of this paper.
\begin{theorem}\label{theorem:goodwillie-approximation}
Let $X$ be a finite CW-complex. Let $E \in \catS_{X//X}$ be a retractive space over $X$. The Goodwillie tower of the functor
\begin{align*}
    F \colon \catS_{E//E} &\lra \catSp \\
        Y &\longmapsto \Sigma^\infty\Gamma(X;Y)
\end{align*}
converges for retractive spaces $Y \in \catS_{E//E}$ such that $\conn(Y/E) > \dim(X)$. Furthermore, the $n$th polynomial approximation is given by the functor $A_n(X,-)$ constructed in \cref{definition:An}. When $X$ is a closed manifold, the $n$th derivative is given by the parameterised spectrum
\[
    \mathrm{ev}_!\bS^0 \otimes p^*\bS^{-TX^n} \otimes p^*\omega_!\bS^0 \in \catSp^{E^n \sslash \Sigma_n}
\]
where $p \colon E^n \sslash \Sigma_n \to X^n \sslash \Sigma_n$ is the projection, $\mathrm{ev} \colon X^n \sslash \Sigma_n \times \Gamma(E) \to E^n \sslash \Sigma_n$ is the evaluation map, and $\omega \colon \UConf_n(X) \to X^n \sslash \Sigma_n$ is the inclusion of the unordered configuration space.
\end{theorem}
\begin{proof}
By the universal property of the $n$th polynomial approximation $P_nF$ \cite[Theorem~1.8]{Goodwillie3} and \cref{proposition:second-variable-excision}, the natural map $F \to A_n(X,-)$ factors through $F \to P_nF$. By \cref{proposition:order-n-agreement} the map $F \to A_n(X,-)$ is an order $n$ agreement, hence $P_nF \to A_n(X,-)$ is an equivalence \cite[Proposition~1.6]{Goodwillie3}. That proposition also shows that the connectivity of the map $F(Y) \to A_n(X,Y) \simeq P_nF(Y)$ tends to infinity with $n$ provided that $\conn(Y/E) > \dim(X)$, hence the Goodwillie tower converges on those spaces $Y$. The formula for the derivative when $X$ is a manifold was given in \cref{corollary:homogeneity}.
\end{proof}

\begin{remark}\label{remark:tower-extra-basepoint}
Instead of basing the space $\Gamma(X;Y)$ at the canonical section $X \to E \to Y$, we could have chosen to add a basepoint and study the functor $Y \mapsto \Sigma^\infty_+\Gamma(X;Y)$. This has little effect on the Goodwillie tower. Indeed there is a natural equivalence $\Sigma^\infty_+\Gamma(X;Y) \simeq \bS^0 \oplus \Sigma^\infty\Gamma(X;Y)$, hence only the zeroth derivatives differ by an $\bS^0$ summand. More directly, one may also modify \cref{definition:An} and work with $\widetilde{A}_n^{k,+} := \widetilde{A}_n^k \vee S^k$ (compare \cite[Remark~3.5]{Malkiewich}).
\end{remark}

\begin{remark}\label{remark:tower-subspace}
More generally, our proof of \cref{theorem:goodwillie-approximation} shows that for a finite CW-complex $B$ with a map $f \colon B \to X$, the $n$th polynomial approximation of the functor
\begin{align*}
    \catS_{E//E} &\lra \catSp \\
    Y &\longmapsto \Sigma^\infty\left( \Gamma(B;f^*Y) \times_{\Gamma(B;f^*E)} \Gamma(X;E) \right)
\end{align*}
is $A_n(B,-)$.
\end{remark}

\section{Comparing the moduli space to the section space}\label{section:comparison}

In this section, $X$ is a smooth projective complex variety and $\cL$ is a very ample line bundle on $X$. Recall that we are interested in the moduli of hypersurfaces equipped with a section of a fibration $Y \to X$ defined in \cref{eqn:def-moduli-Y}, or rather the more technically favourable fattened up version
\[
    \cM^\mathrm{tub}(\cL;Y) = \set{(s,g)}{s \in \cM(\cL), \ g \in \Gamma(Y|_{\tub V(s)})},
\]
introduced in \cref{subsection:scanning}.
Our goal in this section is to prove \cref{theorem:main} by showing that the suspension of the map~\eqref{eqn:main-map}
\begin{equation}\label{eqn:main-natural-transformation}
    \Sigma^\infty_+ \cM^\mathrm{tub}(\cL;Y) \lra \Sigma^\infty_+\Gamma\left(X; (N \sqcup S(J^1\cL)) \wedge_{S(J^1\cL)} p^*(S^\cL \wedge_X Y)\right)
\end{equation}
is an equivalence in a range of degrees. We will do this by first computing the Goodwillie towers in the $Y$ variable on both sides, and then showing that their derivatives agree in a range of degrees.

\subsection{Conventions and notations}

Let $N$ be the tubular neighbourhood chosen once and for all in \cref{subsection:scanning}. Recall that for any $s \in \cM(\cL)$, we write
\[
    \tub V(s) = (\{s\} \times X) \cap \Psi^{-1}(N) = \set{x \in X}{\Psi(s,x) \in N}
\]
for the associated closed tubular neighbourhood of $V(s) \subset X$. (We need to work with the \emph{closed} neighbourhood to ensure that $\tub V(s)$ is a finite CW complex.) Let $r \colon \tub V(s) \to V(s)$ be the retraction of the tubular neighbourhood and define
\[
    \Conf_n^{\mathrm{tub}}(V(s)) = \set{(x_1,\dots,x_n) \in \Conf_n(\tub V(s))}{r(x_i) \neq r(x_j), \ i \neq j}
\]
for the subspace of those configurations in the tubular neighbourhood which retract to configurations on the hypersurface. This is a bundle of products of closed discs over $V(s)$, so the radial contraction $\Conf_n^{\mathrm{tub}}(V(s)) \to V(s)$ is a proper homotopy equivalence. We will also need a fattening of the universal hypersurface bundle with $n$ distinct marked points, so we define
\[
    \cZ^\mathrm{tub}_n(\cL) = \set{(s, c) \in \cM(\cL) \times \Conf_n(X)}{c \subset \Conf_n^{\mathrm{tub}}(V(s))} \lra \cM(\cL).
\]
In the course of the proof, we will perform several base changes. To orient ourselves, we give names to maps in the following diagram:
\[
\begin{tikzcd}
         &                                                           &                                      & X^n \times \Gamma(S(J^1\cL)) \arrow[d, "\mathrm{ev}"] \\
         & \cZ^\mathrm{tub}_n(\cL) \arrow[rd, "q"] \arrow[ld, "\pi"'] &                                      & S(J^1\cL)^n \arrow[d, "p"]                            \\
\cM(\cL) &                                                           & \Conf_n(X) \arrow[r, "\omega", hook] & X^n                                                  
\end{tikzcd}
\]
Here $p$, $q$, $\pi$ are the obvious projections, $\omega$ is the inclusion, and $\mathrm{ev}$ is the evaluation map. We do not further decorate our maps; it will be clear from context which integer $n \geq 0$ is meant. Likewise, we will not distinguish between a map and the map it induces on the quotient by the symmetric group action, e.g. we write $p$ for both $S(J^1\cL)^n \to X^n$ and $S(J^1\cL)^n \sslash \Sigma_n \to X^n \sslash \Sigma_n$. Finally, we will also use the same notation when we restrict a map to a subspace, e.g. $p$ will denote both projections $S(J^1\cL)^n \to X^n$ and $S(J^1\cL)^n|_{\Conf_n(X)} \to \Conf_n(X)$ (as well as the map induced on the quotients).

Let us also recall from \cref{subsection:recollections} that we write $(-)^{\boxtimes n}$ for the external tensor power of a parameterised spectrum, and $\Sigma^k_{B,+}(-) = \Sigma^k_B(- \sqcup B)$ for the fibrewise $k$-fold suspension over $B$ with an added basepoint fibrewise. We also recall that we write $\colim_B(-)$ for the colimit of a parameterised spectrum over $B$, i.e. the functor $r_!(-)$ where $r \colon B \to *$ is the unique map. For a vector bundle $V \to B$, we will often write $\bS^{nV} = (\bS^V)^{\boxtimes n}$ for the $n$-fold tensor product seen as a parameterised spectrum over $B^n$ with a $\Sigma_n$ action.

\subsection{Derivatives of the moduli space}

We start by studying the Goodwillie tower of the functor
\[
    \catS_{X//X} \lra \catSp, \quad Y \longmapsto \Sigma^\infty_+\cM^\mathrm{tub}(\cL;Y).
\]
By definition, the projection $\cM^\mathrm{tub}(\cL;Y) \to \cM(\cL)$ is a fibration with fibre $\Gamma(Y|_{\tub V(s)})$. Hence it straightens to a functor $\cM(\cL) \to \catS$ informally given by $s \mapsto \Gamma(Y|_{\tub V(s)})$, and whose colimit is $\cM^\mathrm{tub}(\cL;Y)$. After applying the functor $\Sigma^\infty_+ \colon \catS \to \catSp$ and taking the colimit, we obtain that $\Sigma^\infty_+\cM^\mathrm{tub}(\cL;Y) \simeq \colim_{s \in \cM(\cL)} \Sigma^\infty_+\Gamma(Y|_{\tub V(s)})$. In \cref{section:goodwillie} we have studied the Goodwillie tower of $Y \mapsto \Sigma^\infty_+\Gamma(Y|_{\tub V(s)})$. It remains to explain how to adapt the formulas given there to the parameterised setting over $\cM(\cL)$.

Let $Y \to X$ be an ex-fibration. We write (cf. \cref{remark:tower-extra-basepoint})
\[
    A^{k,+}_n(-,Y) = \widetilde{A}^k_n(-,Y) \vee S^k
\]
where $\widetilde{A}^k_n(-,Y)$ is defined as in \cref{definition:An} (taking $E=X$ there). The construction $A^{k,+}_n(-,Y)$ is functorial on the category of finite CW complexes with a map to $X$. In particular, we may apply it fibrewise on the fattened universal hypersurface bundle $\cZ^\mathrm{tub}_1(\cL) \to \cM(\cL)$, using that each fibre $\tub V(s) \hookrightarrow X$ comes with a canonical embedding inside $X$. This yields parameterised spaces over $\cM(\cL)$ for each $k \geq 0$
\[
    A^{k,+}_n(\cZ^\mathrm{tub}_1(\cL), Y) = \set{(s, a)}{s \in \cM(\cL), \ a \in A^{k,+}_n(\tub V(s),Y)},
\]
thus defining a parameterised spectrum $A^+_n(\cZ^\mathrm{tub}_1(\cL), Y)$ over $\cM(\cL)$. The formula given in \cref{eqn:adjointmap} (extended to the added basepoint) yields as in \cref{eqn:goodwillie-tower-map} a map of sequential spectra
\begin{equation}\label{eqn:Vs-tower-map}
    \Sigma^k_+\Gamma(Y|_{\tub V(s)}) \lra A^{k,+}_n(\tub V(s), Y), \quad k \geq 0.
\end{equation}
By naturality, these maps assemble into a map of parameterised sequential spectra
\begin{equation}\label{eqn:moduli-tower-map}
    \Sigma^k_{\cM(\cL),+} \cM^\mathrm{tub}(\cL;Y) \lra A^{k,+}_n(\cZ^\mathrm{tub}_1(\cL), Y), \quad k \geq 0.
\end{equation}
The following result is then a direct consequence of our \cref{theorem:goodwillie-approximation}:
\begin{proposition}\label{proposition:parameterised-tower}
Let $n \geq 0$ be an integer. The natural transformation of functors
\[
    \Sigma^\infty_+ \cM^\mathrm{tub}(\cL;-) \lra \colim_{\cM(\cL)}A^+_n(\cZ^\mathrm{tub}_1(\cL),-),
\]
obtained as the colimit over $\cM(\cL)$ of the map~\eqref{eqn:moduli-tower-map}, is the $n$th polynomial approximation in the sense of Goodwillie.
\end{proposition}
\begin{proof}
For spectrum valued functors, the universal polynomial approximation $P_n(-)$ commutes with all colimits \cite[Proposition~1.7.(5)]{Goodwillie3}. Hence it suffices to check that the map~\eqref{eqn:moduli-tower-map} of parameterised spectra
\[
    \Sigma^\infty_{\cM(\cL),+} \cM^\mathrm{tub}(\cL;-) \lra A^+_n(\cZ^\mathrm{tub}_1, -)
\]
is the $n$th polynomial approximation on each fibre. Between the fibres over $s \in \cM(\cL)$ it is the map~\eqref{eqn:Vs-tower-map}
\[
    \Sigma^\infty_+\Gamma(\tub V(s);-) \lra A^+_n(\tub V(s), -)
\]
and the result follows exactly from \cref{theorem:goodwillie-approximation} and \cref{remark:tower-extra-basepoint,remark:tower-subspace}.
\end{proof}

Recall that $\pi \colon \cZ^\mathrm{tub}_n \to \cM(\cL)$ and $q \colon \cZ^\mathrm{tub}_n(\cL) \to \Conf_n(X)$ denote the projections. We write $\pi_!^\sh(-)$ for the exceptional pushforward functor on spectral sheaves \cite[\S~6]{Volpe}. For any $k \geq 0$, we have a parameterised space
\begin{equation}\label{eqn:fibrewise-compact-bundle}
    \Gamma_c(\cZ^\mathrm{tub}_n(\cL),q^*(\Sigma^k_X Y)^{\ovwedge n}) \lra \cM(\cL)
\end{equation}
obtained by fibrewise taking compactly supported sections of $q^*(\Sigma^k_X Y)^{\ovwedge n}$. By taking sections, we obtain a locally constant sheaf of spectra
\begin{equation*}
    \cM(\cL) \overset{\text{open}}{\supset} U \longmapsto \Gamma(U;\Gamma_c(\cZ^\mathrm{tub}_n(\cL),q^*(\Sigma^\infty_X Y)^{\boxtimes n})|_U).
\end{equation*}
We recognise it in the following:
\begin{lemma}\label{lemma:exceptional-pushforward-formula}
There is a natural equivalence of sheaves
\[
    \pi_!^\sh\left( q^*(\Sigma^\infty_X Y)^{\boxtimes n}\right) \overset{\simeq}{\lra} \Gamma(-;\Gamma_c(\cZ^\mathrm{tub}_n(\cL),q^*(\Sigma^\infty_X Y)^{\boxtimes n})|_{(-)}).
\]
\end{lemma}
\begin{proof}
The sheaf $\pi_!^\sh\left( q^*(\Sigma^\infty_X Y)^{\boxtimes n}\right)$ is uniquely determined \cite[Definition~6.1 and the paragraph below]{Volpe} by the formula
\[
    \Gamma_c\left(U;\pi_!^\sh\left( q^*(\Sigma^\infty_X Y)^{\boxtimes n}\right)\right) = \Gamma_c\left(\pi^{-1}(U); q^*(\Sigma^\infty_X Y)^{\boxtimes n}\right).
\]
There is an obvious map obtained by currying
\[
    \Gamma_c(\pi^{-1}(U); q^*(\Sigma^\infty_X Y)^{\boxtimes n}) \lra \Gamma_c(U;\Gamma_c(\cZ^\mathrm{tub}_n(\cL),q^*(\Sigma^\infty_X Y)^{\boxtimes n})|_U),
\]
which, by proper base change \cite[Proposition~6.9]{Volpe}, restricts on the stalk at $s \in \cM(\cL)$ to the identity
\[
    \Gamma_c(\Conf_n^{\mathrm{tub}}(V(s)); q^*(\Sigma^\infty_X Y)^{\boxtimes n}) \overset{\simeq}{\lra} \Gamma_c(\Conf_n^{\mathrm{tub}}(V(s)); q^*(\Sigma^\infty_X Y)^{\boxtimes n}). \qedhere
\]
\end{proof}

Our work on Goodwillie calculus uses configuration spaces rather than fattened up versions, so we shall need to compare them.
\begin{lemma}\label{lemma:configuration-tube-replacement}
Let $\cF$ be a locally constant sheaf of spectra on $(\tub V(s))^n$. Extension by zero along the open inclusion $j \colon \Conf_n^{\mathrm{tub}}(V(s))\hookrightarrow \Conf_n(\tub V(s))$ induces
an equivalence
\[
    \Gamma_c(\Conf_n^{\mathrm{tub}}(V(s));\cF) \overset{\simeq}{\lra} \Gamma_c(\Conf_n(\tub V(s));\cF).
\]
\end{lemma}
\begin{proof}
This follows from the localisation sequence of the inclusion
\[
    \Conf_n^{\mathrm{tub}}(V(s)) \hookrightarrow \Conf_n(\tub V(s)).
\]
Indeed, let $r \colon (\tub V(s))^n \to V(s)^n$ be the radial retraction, and $\Delta \subset V(s)^n$ and $\Delta_\mathrm{tub} \subset (\tub V(s))^n$ be the fat diagonals. Then $\Delta_\mathrm{tub} \hookrightarrow r^{-1}(\Delta)$ is a homotopy equivalence as both spaces radially retract onto $\Delta$. Then, by the localisation sequence, the cofibre of the map of the lemma is
\[
    \Gamma_c(r^{-1}(\Delta) - \Delta_\mathrm{tub}; \cF) \simeq \fib\left(\Gamma(r^{-1}(\Delta);\cF) \overset{\simeq}{\to} \Gamma(\Delta_\mathrm{tub};\cF)\right) \simeq 0,
\]
by homotopy invariance of global sections of locally constant sheaves (i.e. parameterised spectra).
\end{proof}

The sheaf $\pi_!^\sh\left( q^*(\Sigma^\infty_X Y)^{\boxtimes n}\right)$ is locally constant, and thus corresponds to a parameterised spectrum on $\cM(\cL)$. We can in particular make sense of its colimit over $\cM(\cL)$, and use it to compute the differential of our functor:
\begin{corollary}\label{corollary:differential-is-colimit}
Let $n \geq 1$ be an integer. The $n$th differential of the functor $\Sigma^\infty_+\cM^\mathrm{tub}(\cL;-)$ is
\[
    D_n(\Sigma^\infty_+\cM^\mathrm{tub}(\cL;-))(Y) \simeq \colim_{\cM(\cL)} \ \pi_!^\sh\left( q^*(\Sigma^\infty_X Y)^{\boxtimes n}\right)_{h\Sigma_n},
\]
where the symmetric group $\Sigma_n$ acts by permuting the factors of $(\Sigma^\infty_X Y)^{\boxtimes n}$ and the ordering of the configurations.
\end{corollary}
\begin{proof}
For spectrum valued functors, the universal polynomial approximation $P_n(-)$ commutes with all colimits \cite[Proposition~1.7.(5)]{Goodwillie3}, hence so does $D_n(-) = \fib(P_n(-) \to P_{n-1}(-))$. Thus, by \cref{proposition:parameterised-tower} we have that
\[
    D_n(\Sigma^\infty_+\cM^\mathrm{tub}(\cL;-))(Y) \simeq \colim_{\cM(\cL)} \fib\left(A^+_n(\cZ^\mathrm{tub}_1(\cL),Y) \to A^+_{n-1}(\cZ^\mathrm{tub}_1(\cL),Y)\right).
\]
Seen as a locally constant sheaf, the parameterised spectrum $A^+_n(\cZ^\mathrm{tub}_1(\cL),Y)$ is the sheaf
\[
    \cM(\cL) \overset{\text{open}}{\supset} U \longmapsto \Gamma(U;A^+_n(\cZ^\mathrm{tub}_1(\cL),Y)|_U).
\]
Hence, using \cref{lemma:layer-computation} and the equivalence of \cref{lemma:configuration-tube-replacement}, we identify
\[
    \fib(A^+_n(\cZ^\mathrm{tub}_1(\cL),Y) \to A^+_{n-1}(\cZ^\mathrm{tub}_1(\cL),Y))
\]
with the locally constant sheaf
\[
    \cM(\cL) \overset{\text{open}}{\supset} U \longmapsto \Gamma(U;\Gamma_c(\cZ^\mathrm{tub}_n(\cL),q^*(\Sigma^\infty_X Y)^{\boxtimes n})_{h\Sigma_n}|_U).
\]
The result then follows by \cref{lemma:exceptional-pushforward-formula}.
\end{proof}

\subsection{Derivatives of the section space}

We now study the functor
\begin{equation*}
    \catS_{X//X} \lra \catSp, \quad Y \longmapsto \Sigma^\infty_+\Gamma\left(X; (N \sqcup S(J^1\cL)) \wedge_{S(J^1\cL)} p^*(S^\cL \wedge_X Y)\right).
\end{equation*}
We may factorise it as the composition of
\begin{equation}\label{eqn:precomposition}
    \gamma \colon \catS_{X//X} \lra \catS_{S(J^1\cL)//S(J^1\cL)}, \quad Y \longmapsto (N \sqcup S(J^1\cL)) \wedge_{S(J^1\cL)} p^*(S^\cL \wedge_X Y)
\end{equation}
and $\Sigma^\infty_+\Gamma(X;-) \colon \catS_{S(J^1\cL)//S(J^1\cL)} \to \catSp$. Notice that we have already determined the Goodwillie tower of the second functor in \cref{theorem:goodwillie-approximation}. And in fact, precomposition with the first functor only mildly affects the tower:
\begin{lemma}\label{lemma:pulled-back-tower}
Let $\gamma$ be the functor of \cref{eqn:precomposition}. For any finitary functor $G \colon \catS_{S(J^1\cL)//S(J^1\cL)} \to \catSp$, let $\gamma^*G$ be its composition with $\gamma$. If $G$ is $n$-polynomial, resp. $n$-homogeneous, then so is $\gamma^*G$. Furthermore, if $G$ is stably $n$-excisive \cite[Definition~4.1]{Goodwillie2} then $\gamma^*P_nG = P_n(\gamma^*G)$ and the $n$th Goodwillie derivative of $\gamma^*G$ is given by
\[
    \partial_n(\gamma^*G) = p_!\left(\partial_nG \otimes (\Sigma^\infty_{S(J^1\cL),+}N)^{\boxtimes n}\right) \otimes \bS^{n\cL} \in \catSp^{X^n\sslash \Sigma_n}.
\]
\end{lemma}
\begin{proof}
Suppose first that $G$ is homogeneous, hence given by the formula
\[
    G(Z) \simeq \colim_{S(J^1\cL)^n \sslash \Sigma_n}\left(\partial_nG \otimes (\Sigma^\infty_{S(J^1\cL)}Z)^{\boxtimes n}\right).
\]
Then 
\begin{align*}
    \gamma^*G(Y) &= G\left((N \sqcup S(J^1\cL)) \wedge_{S(J^1\cL)} p^*(S^\cL \wedge_X Y)\right) \\
    &\simeq \colim_{S(J^1\cL)^n \sslash \Sigma_n} \left(\partial_nG \otimes \left(\Sigma^\infty_{S(J^1\cL),+}N \otimes p^*(\bS^\cL \otimes \Sigma^\infty_{X}Y)\right)^{\boxtimes n}\right) \\
    &\simeq \colim_{X^n \sslash \Sigma_n} \left(p_!\left(\partial_nG \otimes (\Sigma^\infty_{S(J^1\cL),+}N)^{\boxtimes n}\right) \otimes \bS^{n\cL} \otimes (\Sigma^\infty_{X}Y)^{\boxtimes n}\right)
\end{align*}
where the last equivalence follows from the projection formula. Hence $\gamma^*G$ is visibly $n$-homogeneous with derivative
\[
    p_!\left(\partial_nG \otimes (\Sigma^\infty_{S(J^1\cL),+}N)^{\boxtimes n}\right) \otimes \bS^{n\cL} \in \catSp^{X^n\sslash \Sigma_n}.
\]
The functor $\gamma^*(-)$ preserves fibre sequences, as these are defined pointwise in the functor category. Hence, for any $n \geq 1$ we have a fibre sequence $\gamma^*D_nG \to \gamma^*P_nG \to \gamma^*P_{n-1}G$. By induction on $n$ we see that if $G$ is polynomial, then so is $\gamma^*G$.

Suppose now that $G$ is stably $n$-excisive. As $\gamma^*P_nG$ is $n$-polynomial, the universal property of $P_n(\gamma^*G)$ \cite[Theorem~1.8]{Goodwillie3} shows that the map $\gamma^*G \to \gamma^*P_nG$ factors through $P_n(\gamma^*G)$. As $G$ is stably $n$-excisive, $G \to P_nG$ is an order $n$ agreement \cite[Proposition~1.5]{Goodwillie3}, hence so is $\gamma^*G \to \gamma^*P_nG$ because $\gamma(Y) \to S(J^1\cL)$ is at least as connected as $Y \to X$. Hence $P_n(\gamma^*G) \to P_n(\gamma^*P_nG)$ is an equivalence \cite[Proposition~1.6]{Goodwillie3}. But $\gamma^*P_nG$ is $n$-polynomial, so $P_n(\gamma^*P_nG) \simeq \gamma^*P_nG$.
\end{proof}

\begin{corollary}\label{corollary:differential-sections}
Let $n \geq 1$ be an integer. The $n$th differential of the functor
\[
    \catS_{X//X} \lra \catSp, \quad Y \longmapsto \Sigma^\infty_+\Gamma\left(X; (N \sqcup S(J^1\cL)) \wedge_{S(J^1\cL)} p^*(S^\cL \wedge_X Y)\right)
\]
is given by
\[
    Y \longmapsto \Gamma_c\left(\Conf_n(X); p_!\big(\mathrm{ev}_!\bS^0 \otimes (\Sigma^\infty_{S(J^1\cL),+}N)^{\boxtimes n}\big) \otimes \bS^{n\cL} \otimes (\Sigma^\infty_X Y)^{\boxtimes n}\right)_{h\Sigma_n}.
\]
\end{corollary}
\begin{proof}
This follows directly from \cref{theorem:goodwillie-approximation,lemma:pulled-back-tower}, using that the section space functor $\Sigma^\infty\Gamma(X;-)$ is stably $n$-excisive for all $n$ (see \cite[Example~4.5]{Goodwillie2} for a proof in the very similar case of mapping spaces).
\end{proof}

\subsection{Comparison of derivatives}

In this section, we analyse the natural transformation~\eqref{eqn:main-map}
\[
    \Sigma^\infty_+ \cM^\mathrm{tub}(\cL;Y) \lra \Sigma^\infty_+\Gamma\left(X; (N \sqcup S(J^1\cL)) \wedge_{S(J^1\cL)} p^*(S^\cL \wedge_X Y)\right)
\]
via the induced map between the $n$th differentials. 

We first start with some notation, and then construct a concrete model for the induced map on differentials. When restricted to the configuration space, the evaluation map
\[
    \mathrm{ev} \colon \Conf_n(X) \times \Gamma(S(J^1\cL)) \to S(J^1\cL)^n|_{\Conf_n(X)}
\]
is a fibre bundle, see \cite[Lemma~3.9]{withronno}. Hence $\mathrm{ev}_!S^0$ is a retractive space over $S(J^1\cL)^n|_{\Conf_n(X)}$ with fibre equivalent to the strict fibre of $\mathrm{ev}$ with an added basepoint (cf. \cref{subsection:recollections}). Therefore the retractive space
\[
    p_!\big(\mathrm{ev}_!S^0 \wedge (N \sqcup S(J^1\cL))^{\ovwedge n}\big) \lra \Conf_n(X)
\]
can be identified with the space
\begin{equation}\label{eqn:point-set-identification}
    \set{(c,f,v) \in \Conf_n(X) \times \Gamma(S(J^1\cL)) \times N^n}{f(x_i) = v_i, \ i=1,\dotsc,n } \sqcup \Conf_n(X),
\end{equation}
where we write $c = (x_1,\dotsc,x_n)$ and $v = (v_1,\dotsc,v_n)$. We can thus define a map of bundles:
\begin{equation*}
\begin{split}
    \Gamma_c(\cZ^\mathrm{tub}_n(\cL),q^*(\Sigma^k_X Y)^{\ovwedge n}) &\lra \Gamma_c\left(\Conf_n(X); p_!\big(\mathrm{ev}_!S^0 \wedge (N \sqcup S(J^1\cL))^{\ovwedge n}\big) \wedge S^{n\cL} \wedge (\Sigma^k_X Y)^{\ovwedge n}\right) \\
    (s,g) &\longmapsto \left[ c \mapsto \begin{cases}
                                        (c, \Psi(s,-), \Psi(s,c)) \wedge \can(s,c) \wedge g(c) &\text{if } c \in \Conf_n^{\mathrm{tub}}(V(s))\\
                                        \text{fibrewise basepoint} &\text{else}
                                    \end{cases}\right]
\end{split}
\end{equation*}
Recall that the source was defined in \cref{eqn:fibrewise-compact-bundle}, and the target is seen as a trivial (product) bundle.
In that formula, we have written $(c, \Psi(s,-), \Psi(s,c)) \in p_!\big(\mathrm{ev}_!S^0 \wedge (N \sqcup S(J^1\cL))^{\ovwedge n}\big)$ using our identification~\eqref{eqn:point-set-identification}. And for $c = (x_1,\dotsc,x_n)$, we write $\can(s,c) = \can(s,x_1) \wedge \dots \wedge \can(s,x_n) \in S^{n\cL}$ for the natural extensions of the canonical section~\eqref{eqn:can-section} to the configuration spaces. Taking sections, this yields a map of locally constant sheaves of spectra (cf. \cref{lemma:exceptional-pushforward-formula})
\begin{equation*}
    \pi_!^\sh\left( q^*(\Sigma^\infty_X Y)^{\boxtimes n}\right) \lra \operatorname{cst} \Gamma_c\left(\Conf_n(X); p_!\big(\mathrm{ev}_!\bS^0 \otimes (\Sigma^\infty_{S(J^1\cL),+}N)^{\boxtimes n}\big) \otimes \bS^{n\cL} \otimes (\Sigma^\infty_X Y)^{\boxtimes n}\right)
\end{equation*}
where $\operatorname{cst}(-)$ denotes the constant sheaf on $\cM(\cL)$ associated to a spectrum. By the adjunction $\colim_{\cM(\cL)}(-) \dashv \mathrm{cst}(-)$, we obtain an adjoint map:
\begin{equation}\label{eqn:differential-map-adjoint}
    \colim_{\cM(\cL)} \ \pi_!^\sh\left( q^*(\Sigma^\infty_X Y)^{\boxtimes n}\right) \to \Gamma_c\left(\Conf_n(X); p_!\big(\mathrm{ev}_!\bS^0 \otimes (\Sigma^\infty_{S(J^1\cL),+}N)^{\boxtimes n}\big) \otimes \bS^{n\cL} \otimes (\Sigma^\infty_X Y)^{\boxtimes n}\right).
\end{equation}
\begin{lemma}\label{lemma:map-on-differentials}
Let $n \geq 1$. Under the identifications of \cref{corollary:differential-is-colimit,corollary:differential-sections}, taking $\Sigma_n$-homotopy orbits on both sides of the map~\eqref{eqn:differential-map-adjoint} gives the map
\[
    D_n\Sigma^\infty_+ \cM^\mathrm{tub}(\cL;Y) \lra D_n\Sigma^\infty_+\Gamma\left(X; (N \sqcup S(J^1\cL)) \wedge_{S(J^1\cL)} p^*(S^\cL \wedge_X Y)\right)
\]
on $n$th differentials induced by the natural transformation~\eqref{eqn:main-natural-transformation}.
\end{lemma}
\begin{proof}
This follows from chasing through our definitions using the concrete formulas for our scanning map~\eqref{eqn:scanning-map} and the Goodwillie approximation~\eqref{eqn:adjointmap}, together with our computations of the differentials in \cref{corollary:differential-is-colimit,corollary:differential-sections}.
\end{proof}

To really compare the differentials, we shall rewrite the spectrum
\[
    \colim_{\cM(\cL)} \ \pi_!^\sh\left( q^*(\Sigma^\infty_X Y)^{\boxtimes n}\right)_{h\Sigma_n} \simeq D_n\Sigma^\infty_+ \cM^\mathrm{tub}(\cL;Y)
\]
in terms of the configuration space $\Conf_n(X)$. Let us factor the projection
\[
    \cZ^\mathrm{tub}_n(\cL) \sqcup \Conf_n(X) \to \Conf_n(X)
\]
through an acyclic cofibration $\eta \colon \cZ^\mathrm{tub}_n(\cL) \sqcup \Conf_n(X) \overset{\simeq}{\to} q_!S^0$ such that $q_!S^0 \to \Conf_n(X)$ is an ex-fibration. (We do this so that our later section spaces have the correct homotopy type.) There is a map of bundles
\begin{equation*}
\begin{split}
    \Gamma_c(\cZ^\mathrm{tub}_n(\cL),q^*(\Sigma^k_X Y)^{\ovwedge n}) &\lra \Gamma_c\left(\Conf_n(X); q_!S^0 \wedge S^{n\cL} \wedge (\Sigma^k_X Y)^{\ovwedge n}\right) \\
    (s,g) &\longmapsto \left[ c \mapsto \begin{cases}
                        \eta(s,c) \wedge \can(s,c) \wedge g(c) &\text{if } c \in \Conf_n^{\mathrm{tub}}(V(s))\\
                        \text{fibrewise basepoint} &\text{else}
                    \end{cases} \right]
\end{split}
\end{equation*}
where the target is seen as a trivial bundle over $\cM(\cL)$. As above, we thus obtain a map of locally constant sheaves of spectra on $\cM(\cL)$
\begin{equation}\label{eqn:sheaf-to-conf}
    \pi_!^\sh\left( q^*(\Sigma^\infty_X Y)^{\boxtimes n}\right) \lra \operatorname{cst}\Gamma_c\left(\Conf_n(X); q_!\bS^0 \otimes \bS^{n\cL} \otimes (\Sigma^\infty_X Y)^{\boxtimes n}\right)
\end{equation}
where the target spectrum is interpreted as a constant sheaf.
\begin{lemma}\label{lemma:basechange-rewriting}
The adjoint of the map~\eqref{eqn:sheaf-to-conf}
\[
    \colim_{\cM(\cL)} \ \pi_!^\sh\left( q^*(\Sigma^\infty_X Y)^{\boxtimes n}\right) \lra \Gamma_c\left(\Conf_n(X); q_!\bS^0 \otimes \bS^{n\cL} \otimes (\Sigma^\infty_X Y)^{\boxtimes n}\right)
\]
is an equivalence.
\end{lemma}
\begin{proof}
Formal manipulations with sheaves show that the spectra in the lemma are abstractly equivalent. However, as we want to keep track of the maps involved, we shall present a slightly cumbersome argument. 
Let us start by giving names to maps in the following commutative diagram:
\[
\begin{tikzcd}
\cZ_n(\cL) \arrow[r, hook, "\iota"] & \cZ^\mathrm{tub}_n(\cL) \arrow[r, hook, "i"] \arrow[d, "\pi"'] & |[label={[label distance=-2mm]-45:\lrcorner}]|\cM(\cL) \times \Conf_n(X) \arrow[d, "\varpi"'] \arrow[r, "\overline{q}"] & \Conf_n(X) \arrow[d, "\overline{r}"] \\
& \cM(\cL) \arrow[r, equal] & \cM(\cL) \arrow[r, "r"']  & *
\end{tikzcd}
\]
We shall also abbreviate $q = \overline{q} \circ i$ and $\tilde{q} = q \circ \iota$.
Smashing a section with $\eta \wedge \can$ gives a morphism
\[
    \Gamma_c\left(i^{-1}(U); q^*(\Sigma^\infty_X Y)^{\boxtimes n}\right) \simeq \Gamma_c\left(U; i^\sh_!q^*(\Sigma^\infty_X Y)^{\boxtimes n}\right) \lra \Gamma_c\left(U; \overline{q}^*(q_!\bS^0 \otimes \bS^{n\cL} \otimes (\Sigma^\infty_X Y)^{\boxtimes n}) \right)
\]
naturally for any open subset $U \subset \cM(\cL) \times \Conf_n(X)$. By \cite[below Definition~6.1]{Volpe}, this defines a morphism of sheaves
\[
    \Phi \colon i^\sh_!q^*(\Sigma^\infty_X Y)^{\boxtimes n} \lra \overline{q}^*\left(q_!\bS^0 \otimes \bS^{n\cL} \otimes (\Sigma^\infty_X Y)^{\boxtimes n}\right).
\]
By construction, the morphism~\eqref{eqn:sheaf-to-conf} is $\varpi^\sh_! \Phi$ after base-change identification. On $\cM(\cL) \times \Conf_n(X) - \cZ_n(\cL)$, the morphism $\Phi$ is null-homotopic with a null-homotopy induced from the fibrewise radial retraction $S^{n\cL} - \{0\} \simeq \{\infty\}$ of the punctured sphere to the point at infinity. It follows from the localisation sequence that $\Phi$ factors through the restriction map to $\cZ_n(\cL)$
\[
    \Phi \colon i^\sh_!q^*(\Sigma^\infty_X Y)^{\boxtimes n} \overset{\mathrm{res}}{\lra} (i\circ \iota)^\sh_!\tilde{q}^*(\Sigma^\infty_X Y)^{\boxtimes n} \overset{\Phi'}{\lra} \overline{q}^*\left(q_!\bS^0 \otimes \bS^{n\cL} \otimes (\Sigma^\infty_X Y)^{\boxtimes n}\right).
\]
Now, the proper radial deformation retraction of the tubular configuration space and local constancy of the coefficient spectrum show that restriction induces a homotopy equivalence
\[
    \pi_!^\sh\left( q^*(\Sigma^\infty_X Y)^{\boxtimes n}\right) \simeq \pi_!^\sh \iota^\sh_!\left( \tilde{q}^*(\Sigma^\infty_X Y)^{\boxtimes n}\right),
\]
hence the map~\eqref{eqn:sheaf-to-conf} is equivalent to the composite
\begin{align*}
    \pi_!^\sh \iota^\sh_!\left( \tilde{q}^*(\Sigma^\infty_X Y)^{\boxtimes n}\right) &\overset{\varpi^\sh_! \Phi'}{\lra} \varpi^\sh_!\overline{q}^*\left(\tilde{q}_!\bS^0 \otimes \bS^{n\cL} \otimes (\Sigma^\infty_X Y)^{\boxtimes n}\right) \\
    &\lra \operatorname{cst}\Gamma_c\left(\Conf_n(X); \tilde{q}_!\bS^0 \otimes \bS^{n\cL} \otimes (\Sigma^\infty_X Y)^{\boxtimes n}\right)
\end{align*}
where the second morphism is the base-change morphism, and we have implicitly used that $q_!\bS^0 \simeq \tilde{q}_!\bS^0$. First, the adjoint of the base-change morphism is
\begin{align*}
    \overline{q}^*\left(\tilde{q}_!\bS^0 \otimes \bS^{n\cL} \otimes (\Sigma^\infty_X Y)^{\boxtimes n}\right) &\overset{\overline{q}^*v}{\lra} \overline{q}^* \overline{r}^! \overline{r}^\sh_!\left(\tilde{q}_!\bS^0 \otimes \bS^{n\cL} \otimes (\Sigma^\infty_X Y)^{\boxtimes n}\right) \\
    &\overset{\simeq}{\lra} \varpi^!r^*\overline{r}^\sh_!\left(\tilde{q}_!\bS^0 \otimes \bS^{n\cL} \otimes (\Sigma^\infty_X Y)^{\boxtimes n}\right)
\end{align*}
where $v \colon \left(\tilde{q}_!\bS^0 \otimes \bS^{n\cL} \otimes (\Sigma^\infty_X Y)^{\boxtimes n}\right) \to \overline{r}^! \overline{r}^\sh_! \left(\tilde{q}_!\bS^0 \otimes \bS^{n\cL} \otimes (\Sigma^\infty_X Y)^{\boxtimes n}\right)$ is the unit map. Second, under the equivalence \cite[Corollary~7.12]{Volpe}
\[
    (i \circ \iota)^!\overline{q}^*\left(\tilde{q}_!\bS^0 \otimes \bS^{n\cL} \otimes (\Sigma^\infty_X Y)^{\boxtimes n}\right) \simeq \tilde{q}^*\left(\tilde{q}_!\bS^0 \otimes (\Sigma^\infty_X Y)^{\boxtimes n}\right),
\]
the adjoint of $\Phi'$ is
\[
    u \otimes \mathrm{id}_{\tilde{q}^*(\Sigma^\infty_X Y)^{\boxtimes n}} \colon \tilde{q}^*(\Sigma^\infty_X Y)^{\boxtimes n} \lra \tilde{q}^*\left(\tilde{q}_!\bS^0 \otimes (\Sigma^\infty_X Y)^{\boxtimes n}\right)
\]
where $u \colon \bS^0 \to \tilde{q}^*\tilde{q}_!\bS^0$ is the unit map. This is the key point of our unpacking the definitions: up to the homotopical replacement $\eta$, the map $\can$ implements a Thom collapse map. Under relative Atiyah duality \cite[\S~7]{Volpe}, the map of the lemma is canonically identified with the map
\[
    \colim_{\cZ_n(\cL)} \tilde{q}^*\left(\bS^{n\cL - T\Conf_n(X)} \otimes (\Sigma^\infty_X Y)^{\boxtimes n}\right) \lra \colim_{\Conf_n(X)} \left(\tilde{q}_!\bS^0 \otimes \bS^{n\cL - T\Conf_n(X)} \otimes (\Sigma^\infty_X Y)^{\boxtimes n}\right)
\]
induced by the adjoint of $u \otimes \mathrm{id}_{\tilde{q}^*(\bS^{n\cL - T\Conf_n(X)} \otimes (\Sigma^\infty_X Y)^{\boxtimes n})}$
\[
    \tilde{q}_! \tilde{q}^* \left(\bS^{n\cL - T\Conf_n(X)} \otimes (\Sigma^\infty_X Y)^{\boxtimes n}\right) \overset{\simeq}{\lra} \tilde{q}_!\bS^0 \otimes \bS^{n\cL - T\Conf_n(X)} \otimes (\Sigma^\infty_X Y)^{\boxtimes n}.
\]
This map is an equivalence by the projection formula.
\end{proof}

Under the identification~\eqref{eqn:point-set-identification}, we may define the following map of spaces over $\Conf_n(X)$:
\begin{align*}
    \cZ^\mathrm{tub}_n(\cL) &\lra p_!\big(\mathrm{ev}_!S^0 \wedge (N \sqcup S(J^1\cL))^{\ovwedge n}\big) \\
    (s,c) &\longmapsto (c, \Psi(s,-), \Psi(s,c))
\end{align*}
By fibrewise suspending, and using the fibrant replacement $\cZ^\mathrm{tub}_n(\cL) \sqcup \Conf_n(X) \simeq q_!S^0$ as above, it yields a map
\begin{equation}\label{eqn:configuration-jet-expansion}
    q_!\bS^0 \lra p_!\big(\mathrm{ev}_!\bS^0 \otimes (\Sigma^\infty_{S(J^1\cL),+}N)^{\boxtimes n}\big).
\end{equation}
By combining \cref{lemma:map-on-differentials,lemma:basechange-rewriting} we finally obtain:
\begin{proposition}\label{proposition:final-identification-map-differentials}
Let $n \geq 1$. The map on the $n$th Goodwillie derivatives
\[
    D_n\left(\Sigma^\infty_+ \cM^\mathrm{tub}(\cL;Y)\right) \lra D_n\left(\Sigma^\infty_+\Gamma\left(X; (N \sqcup S(J^1\cL)) \wedge_{S(J^1\cL)} p^*(S^\cL \wedge_X Y)\right)\right)
\]
can be identified with the morphism
\begin{align*}
    &\Gamma_c\left(\Conf_n(X); q_!\bS^0 \otimes \bS^{n\cL} \otimes (\Sigma^\infty_X Y)^{\boxtimes n}\right)_{h\Sigma_n} \\
    &\lra \Gamma_c\left(\Conf_n(X); p_!\big(\mathrm{ev}_!\bS^0 \otimes (\Sigma^\infty_{S(J^1\cL),+}N)^{\boxtimes n}\big) \otimes \bS^{n\cL} \otimes (\Sigma^\infty_X Y)^{\boxtimes n}\right)_{h\Sigma_n}
\end{align*}
induced by the map~\eqref{eqn:configuration-jet-expansion}. \qed
\end{proposition}

Our goal is now to show that the map between the differentials is highly connected. We begin by proving that the map~\eqref{eqn:configuration-jet-expansion} of parameterised spectra is fibrewise highly connected. Let
\[
    \cZ_n^{\cC^0}(\cL) = {\set{(c,f) \in \Conf_n(X) \times \Gamma(S(J^1\cL))}{\forall x \in c, \ f(x) \in S(\Omega^1_X \otimes \cL) \subset S(J^1\cL)}}
\]
and
\[
    \cZ^{\mathrm{tub}, \cC^0}_n(\cL) = {\set{(c,f,v) \in \Conf_n(X) \times \Gamma(S(J^1\cL)) \times N^n}{f(x_i) = v_i, \ i=1,\dotsc,n }}
\]
where $c = (x_1,\dotsc,x_n)$ and $v = (v_1,\dotsc,v_n)$.
As the tubular neighbourhood $N$ deformation retracts onto its zero section $S(\Omega^1_X \otimes \cL)$, we obtain homotopy equivalences of spaces in $\catS_{/\Conf_n(X)}$ fitting in the following commutative square:
\begin{equation}\label{eqn:retract-N-square}
\begin{tikzcd}[column sep=8em]
\cZ_n(\cL) \arrow[d, "\simeq"'] \arrow[r, "{(c,s) \mapsto (c, \Psi(s,-))}"]   & {\cZ_n^{\cC^0}(\cL)} \arrow[d, "\simeq"]  \\
\cZ^\mathrm{tub}_n(\cL) \arrow[r, "{(c,s) \mapsto (c, \Psi(s,-), \Psi(s,c))}"'] & {\cZ^{\mathrm{tub}, \cC^0}_n(\cL)}
\end{tikzcd}
\end{equation}
To understand the effect in homology of the top horizontal map, we shall need a result from previous work of the author with Das. In what follows, we say that a map of spaces $A \to B$ is homology $m$-connected if the induced map on homology $H_*(A;\bZ) \to H_*(B;\bZ)$ is surjective in degrees $* \leq m$ and an isomorphism in degrees $* < m$.
\begin{theorem}\label{theorem:hprinciple-with-vanishing}
Let $c \in \Conf_n(X)$. Suppose that $\cL$ is $d$-jet ample. Then the jet map
\[
    j^1 \colon \set{s \in \cM(\cL)}{\forall x \in c, \ s(x) = 0} \lra \set{f \in \Gamma(S(J^1\cL))}{\forall x \in c, \ f(x) \in S(\Omega^1_X \otimes \cL)}
\]
is homology $(\frac{d-3}{2}-n)$-connected.
\end{theorem}
\begin{proof}
Write $c = (x_1,\dotsc,x_n)$. The codomain and domain of the jet map $j^1$ can be viewed as spaces over $\prod_{i = 1}^n S(\Omega^1_X \otimes \cL)_{x_i}$ via the maps
\begin{align*}
    \set{f \in \Gamma(S(J^1\cL))}{\forall x \in c, \ f(x) \in S(\Omega^1_X \otimes \cL)} &\lra \prod_{i = 1}^n S(\Omega^1_X \otimes \cL)_{x_i} \\
    f &\longmapsto (f(x_1),\dotsc,f(x_n))
\end{align*}
and
\begin{align*}
    \set{s \in \cM(\cL)}{\forall x \in c, \ s(x) = 0} &\lra \prod_{i = 1}^n S(\Omega^1_X \otimes \cL)_{x_i} \\
    s &\longmapsto (\mathrm{d}s(x_1),\dotsc,\mathrm{d}s(x_n)).
\end{align*}
The restriction of the jet map to any fibre is homology $(\frac{d-3}{2}-n)$-connected by \cite[Theorem~5.1]{withronno}. We can then argue as in \cite[\S~3.2, Proof of Theorem~3.1]{withronno}: the first map above is a fibre bundle, and the second is a microfibration obtained as the restriction of a fibre bundle to an open subset of its total space. Hence \cite[Lemma~3.8]{withronno} applies and the theorem follows.
\end{proof}

\begin{corollary}\label{corollary:connectivity-fibrewise}
Let $c \in \Conf_n(X)$. Suppose that $\cL$ is $d$-jet ample. Then the restriction of the map~\eqref{eqn:configuration-jet-expansion} to the fibre above $c$
\[
    \big(q_!\bS^0\big)_c \lra \big(p_!\big(\mathrm{ev}_!\bS^0 \otimes (\Sigma^\infty_{S(J^1\cL),+}N)^{\boxtimes n}\big)\big)_c
\]
is $(\frac{d-3}{2}-n)$-connected.
\end{corollary}
\begin{proof}
Taking fibrewise suspensions over $\Conf_n(X)$ in the commutative square~\eqref{eqn:retract-N-square} yields a commutative square of parameterised spectra:
\[
\begin{tikzcd}
\widetilde{q}_!\bS^0 \arrow[r] \arrow[d, "\simeq"'] & {p_!\big(\mathrm{ev}_!\bS^0 \otimes (\Sigma^\infty_{S(J^1\cL),+}S(\Omega^1_X \otimes \cL))^{\boxtimes n}\big)} \arrow[d, "\simeq"] \\
q_!\bS^0 \arrow[r]   & {p_!\big(\mathrm{ev}_!\bS^0 \otimes (\Sigma^\infty_{S(J^1\cL),+}N)^{\boxtimes n}\big)}                                            
\end{tikzcd}
\]
where $\widetilde{q} \colon \cZ_n(\cL) \to \Conf_n(X)$ is the projection. It thus suffices to show that the map
\[
    \big(\widetilde{q}_!\bS^0\big)_c \lra \big(p_!\big(\mathrm{ev}_!\bS^0 \otimes (\Sigma^\infty_{S(J^1\cL),+}S(\Omega^1_X \otimes \cL))^{\boxtimes n}\big)\big)_c
\]
induces an isomorphism on homotopy groups in the claimed range. We may identify this map with
\begin{equation}\label{eqn:map-on-hofib-param-spectra}
    \Sigma^\infty_+ \fib_c\left(\cZ_n(\cL) \to \Conf_n(X)\right) \overset{j^1}{\lra} \Sigma^\infty_+\set{f \in \Gamma(S(J^1\cL))}{\forall x \in c, \ f(x) \in S(\Omega^1_X \otimes \cL)},
\end{equation}
where $\fib_c(-)$ denotes the homotopy fibre at $c$. This is nearly the map of \cref{theorem:hprinciple-with-vanishing}, except for the fact that the domain is the homotopy fibre rather than the strict fibre. However, we can argue as in the proof of that theorem to show that our connectivity estimate still holds. Indeed, we have a map from a microfibration to a fibre bundle:
\[
\begin{tikzcd}[column sep=0em]
{\set{s \in \cM(\cL)}{\forall x \in c, \ s(x) = 0}} \arrow[r] \arrow[d] & \fib_c(\widetilde{q}) \arrow[r] \arrow[dl] & {\set{f \in \Gamma(S(J^1\cL))}{\forall x \in c, \ f(x) \in S(\Omega^1_X \otimes \cL)}} \arrow[d] \\
\cZ_n(\cL) \arrow[rd, "\widetilde{q}"'] \arrow[rr] & & \cZ^{\cC^0}_n(\cL) \arrow[ld] \\
& \Conf_n(X) & 
\end{tikzcd}
\]
After (unreduced) suspending twice, the map
\[
    \Sigma^2\set{s \in \cM(\cL)}{\forall x \in c, \ s(x) = 0} \lra \Sigma^2\set{f \in \Gamma(S(J^1\cL))}{\forall x \in c, \ f(x) \in S(\Omega^1_X \otimes \cL)}
\]
is $(\frac{d-3}{2}-n+2)$-connected by \cref{theorem:hprinciple-with-vanishing}. (As in \cite[\S~3.1]{withronno} the suspensions of the spaces are simply-connected, hence we can upgrade the homological connectivity to a connectivity on homotopy groups.) By \cite[Lemma~3.8]{withronno}, we deduce that
\[
    \Sigma^2_\mathrm{fib}\cZ_n(\cL) \lra \Sigma^2_\mathrm{fib} \cZ_n^{\cC^0}(\cL)
\]
is $(\frac{d-3}{2}-n+2)$-connected, where $\Sigma^2_\mathrm{fib}$ means the fibrewise unreduced suspension. By the 5-lemma, we deduce that
\[
    \Sigma^2\fib_c(\widetilde{q}) \simeq \fib_c(\Sigma^2_\mathrm{fib}\widetilde{q}) \lra \Sigma^2\set{f \in \Gamma(S(J^1\cL))}{\forall x \in c, \ f(x) \in S(\Omega^1_X \otimes \cL)}
\]
is also $(\frac{d-3}{2}-n+2)$-connected. (The first equivalence follows by a direct computation with Mather's second cube theorem \cite[Theorem~25]{Mather1976}.) Hence the map~\eqref{eqn:map-on-hofib-param-spectra} is $(\frac{d-3}{2}-n)$-connected.
\end{proof}

\begin{proposition}\label{proposition:connectivity-map-differentials}
Let $n \geq 1$. Suppose that $\cL$ is $d$-jet ample. Then the map on the $n$th Goodwillie differentials
\[
    D_n\left(\Sigma^\infty_+ \cM^\mathrm{tub}(\cL;Y)\right) \lra D_n\left(\Sigma^\infty_+\Gamma\left(X; (N \sqcup S(J^1\cL)) \wedge_{S(J^1\cL)} p^*(S^\cL \wedge_X Y)\right)\right)
\]
is $(\frac{d-3}{2} + n(\conn(Y/X) + 1 - \dim_\bR(X)))$-connected.
\end{proposition}
\begin{proof}
Poincaré duality for the manifold $\Conf_n(X)$ gives an equivalence
\[
    \Gamma_c(\Conf_n(X); -) \simeq \colim_{\Conf_n(X)}\left(- \otimes \bS^{-T\Conf_n(X)}\right)
\]
for any parameterised spectrum $(-)$. Combined with \cref{proposition:final-identification-map-differentials}, we see that the map on the $n$th Goodwillie differentials has the form
\begin{align*}
    &\colim_{\UConf_n(X)}\left(q_!\bS^0 \otimes \bS^{n\cL} \otimes (\Sigma^\infty_X Y)^{\boxtimes n} \otimes \bS^{-T\Conf_n(X)}\right) \\
    &\lra \colim_{\UConf_n(X)}\left(p_!\big(\mathrm{ev}_!\bS^0 \otimes (\Sigma^\infty_{S(J^1\cL),+}N)^{\boxtimes n}\big) \otimes \bS^{n\cL} \otimes (\Sigma^\infty_X Y)^{\boxtimes n} \otimes \bS^{-T\Conf_n(X)}\right)
\end{align*}
and is induced by the map $q_!\bS^0 \to p_!\big(\mathrm{ev}_!\bS^0 \otimes (\Sigma^\infty_{S(J^1\cL),+}N)^{\boxtimes n}\big)$ given in \cref{eqn:configuration-jet-expansion}. This map is fibrewise $(\frac{d-3}{2}-n)$-connected by \cref{corollary:connectivity-fibrewise}. Furthermore, $\bS^{n\cL}$ is fibrewise $(2n-1)$-connected, $(\Sigma^\infty_X Y)^{\boxtimes n}$ is fibrewise $(n \conn(Y/X)-1)$-connected, and $\bS^{-T\Conf_n(X)}$ is fibrewise $(-n\dim_\bR(X)-1)$-connected. Finally, recall that taking homotopy orbits does not lower the connectivity, and that the tensor product of an $(i-1)$-connected spectrum with a $(j-1)$-connected spectrum is $(i+j-1)$-connected (both statements follow because connective spectra are stable under colimits \cite[Proposition~1.4.3.4]{HA}). By summing up, we thus obtain that the map between the $n$th differentials has connectivity:
\[
    \frac{d-3}{2}-n + 2n + n\conn(Y/X) -n\dim_\bR(X) = \frac{d-3}{2} + n(\conn(Y/X) + 1 - \dim_\bR(X)). \qedhere
\]
\end{proof}

\subsection{Proof of the main theorem}

We are now ready to give the proof of \cref{theorem:main}. The underlying idea is very simple: we write down the Goodwillie towers and examine the connectivity of the comparison map on polynomial approximations. This is done inductively using the connectivity established between differentials.

\begin{proof}[Proof of \cref{theorem:main}]
To lighten the notation, let us write $\Gamma(\cL;Y)$ for the functor 
\[
    Y \longmapsto \Gamma\left(X; (N \sqcup S(J^1\cL)) \wedge_{S(J^1\cL)} p^*(S^\cL \wedge_X Y)\right).
\]
The natural transformation $\Sigma^\infty_+ \cM^\mathrm{tub}(\cL;Y) \to \Sigma^\infty_+ \Gamma(\cL;Y)$ induces a map of Goodwillie towers:
\[
    P_\bullet\left(\Sigma^\infty_+ \cM^\mathrm{tub}(\cL;Y)\right) \lra P_\bullet\left(\Sigma^\infty_+ \Gamma(\cL;Y)\right)    
\]
At the bottom of the tower, the map
\[
    \Sigma^\infty_+j^1 \colon \Sigma^\infty_+\cM(\cL) \simeq P_0\left(\Sigma^\infty_+ \cM^\mathrm{tub}(\cL;Y)\right) \lra P_0\left(\Sigma^\infty_+ \Gamma(\cL;Y)\right) \simeq \Sigma^\infty_+ \Gamma(S(J^1\cL))
\]
is $\frac{d-3}{2}$-connected by \cite[Corollary~8.1]{hprinciple}. 
Recall that we are assuming that 
\[
    \conn(Y/X) \geq \dim_\bR(X)-1,
\]
hence for $n \geq 1$ the comparison map on differentials
\[
    D_n\left(\Sigma^\infty_+ \cM^\mathrm{tub}(\cL;Y)\right)\lra D_n\left(\Sigma^\infty_+ \Gamma(\cL;Y)\right)
\]
is at least $\frac{d-3}{2}$-connected by \cref{proposition:connectivity-map-differentials}. Using inductively the 5-lemma, it follows that each map
\[
    P_n\left(\Sigma^\infty_+ \cM^\mathrm{tub}(\cL;Y)\right) \lra P_n\left(\Sigma^\infty_+ \Gamma(\cL;Y)\right)
\]
is also $\frac{d-3}{2}$-connected. Now $\Sigma^\infty_+\cM^\mathrm{tub}(\cL;Y) \simeq \colim_{s \in \cM(\cL)} \Sigma^\infty_+\Gamma(Y|_{\tub V(s)})$ is the colimit of $(\dim_\bR(X)-2)$-analytic functors, hence is itself $(\dim_\bR(X)-2)$-analytic. Therefore the natural map to the limit of the tower
\[
    \Sigma^\infty_+\cM^\mathrm{tub}(\cL;Y) \lra \lim_{n \to \infty} P_n\left(\Sigma^\infty_+ \cM^\mathrm{tub}(\cL;Y)\right)
\]
is an equivalence when $\conn(Y/X) > \dim_\bR(X) - 2$. On the other hand,
\[
    \Sigma^\infty_+\Gamma(X;-) \colon \catS_{S(J^1\cL)//S(J^1\cL)} \to \catSp
\]
is $\dim_\bR(X)$-analytic (this follows from the argument given in the proof of \cref{proposition:order-n-agreement}). Now, if $Y \to X$ is $k$-connected then $(N \sqcup S(J^1\cL)) \wedge_{S(J^1\cL)} p^*(S^\cL \wedge_X Y) \to S(J^1\cL)$ is $(k+2)$-connected. Hence $\Sigma^\infty_+ \Gamma(\cL;Y)$ is $(\dim_\bR(X)-2)$-analytic, and therefore
\[
    \Sigma^\infty_+ \Gamma(\cL;Y) \lra \lim_{n \to \infty} P_n\left(\Sigma^\infty_+ \Gamma(\cL;Y)\right)
\]
is an equivalence when $\conn(Y/X) > \dim_\bR(X) - 2$.
\end{proof}

\section{An equivalence of fibrations}\label{section:equivalence}

In this section, we show that the fibration
\[
    (N \sqcup S(J^1\cL)) \wedge_{S(J^1\cL)} p^*(S^\cL \wedge_X Y) \lra X
\]
is equivalent to the fibrewise join
\[
    S(\cL) *_X (S(\Omega^1_X \otimes \cL) \times_X Y) \lra X.
\]
It will be useful to consider a third fibration, as a middle ground between the two above. Recall from \cref{subsection:scanning} that $N$ is a closed tubular neighbourhood of $S(\Omega^1_X \otimes \cL) \subset S(J^1\cL)$, that $\partial N = N - \mathring{N}$ is the associated sphere bundle, and that $\rho \colon N \to X$ is the restriction of the projection $p \colon S(J^1\cL) \to X$ to $N$. We first define $\frac{\rho^*Y}{\sim}$ as the pushout
\begin{equation*}
\begin{tikzcd}
\rho|_{\partial N}^*Y \arrow[r, hook] \arrow[d] & \rho^*Y \arrow[d]    \\
\partial N \arrow[r]                      & |[label={[label distance=-15mm]-45:\ulcorner}]|\frac{\rho^*Y}{\sim}
\end{tikzcd}
\end{equation*}
Assuming, as usual, that the fibration $Y \to X$ has a section $u \colon X \to Y$, we let $\rho^*u \colon N \to \frac{\rho^*Y}{\sim}$ be the pulled back section composed with the quotient map. We then define $S(J^1\cL) \cup_N \frac{\rho^*Y}{\sim}$ as the pushout
\begin{equation*}
\begin{tikzcd}
N \arrow[r, "\rho^*u"] \arrow[d, hook] & \frac{\rho^*Y}{\sim} \arrow[d] \\
S(J^1\cL) \arrow[r]                            & |[label={[label distance=-15mm]-45:\ulcorner}]|S(J^1\cL) \cup_N \frac{\rho^*Y}{\sim} 
\end{tikzcd}
\end{equation*}
\begin{lemma}\label{lemma:ansazt-nice-formula}
There is a weak homotopy equivalence of retractive spaces over $S(J^1\cL)$:
\[
    S(J^1\cL) \cup_N \frac{\rho^*Y}{\sim} \overset{\simeq}{\lra} (N \sqcup S(J^1\cL)) \wedge_{S(J^1\cL)} p^*(S^\cL \wedge_X Y).
\]
\end{lemma}

\begin{proof}
We will first construct a candidate map, and then verify that it is an equivalence fibrewise. For clarity in our formulas we write $E = S(J^1\cL)$ for the duration of this proof. Recall from \cref{subsection:scanning} that we have a natural section $\mathrm{vec} \colon N \to \rho^*S^\cL$. It was obtained by identifying the open tubular neighbourhood $\mathring{N}$ with $(p^*\cL)|_{S(\Omega^1_X \otimes \cL)}$, and the section informally records the normal vector (see \cref{figure:tubular-vector}).

We will now construct a map
\[
    S(J^1\cL) \cup_N \frac{\rho^*Y}{\sim} \lra (N \sqcup S(J^1\cL)) \wedge_{S(J^1\cL)} p^*(S^\cL \wedge_X Y).
\]
Let $\widetilde{u} \colon X \to S^\cL \wedge_X Y$ denote the basepoint section. By the universal property of pushouts, it suffices to construct maps on both $\frac{\rho^*Y}{\sim}$ and $S(J^1\cL)$ that agree along $N$. We write $\mathrm{incl}$ and $\mathrm{proj}$ to denote obvious inclusions and projections. The maps
\begin{align*}
    \mathrm{incl} \circ \mathrm{proj} \colon&  \rho^*Y \to N \hookrightarrow N \sqcup E \\
    \mathrm{incl} \circ (\mathrm{vec} \wedge \mathrm{id}) \colon& \rho^*Y \to \rho^*S^\cL \wedge_N \rho^*Y \hookrightarrow p^*(S^\cL \wedge_X Y)
\end{align*}
assemble into a map
\[
    \rho^*Y \lra (N \sqcup E) \wedge_E p^*(S^\cL \wedge_X Y)
\]
into the smash product. Similarly,
\begin{align*}
    \mathrm{incl} \colon& \partial N \hookrightarrow N \hookrightarrow N \sqcup E \\
    p^*\widetilde{u} \circ \mathrm{incl} \colon& \partial N \hookrightarrow E \to p^*(S^\cL \wedge_X Y)
\end{align*}
give a map
\[
    \partial N \lra (N \sqcup E) \wedge_E p^*(S^\cL \wedge_X Y).
\]
One readily checks that these maps agree on $\rho|_{\partial N}^*Y$ (they map to the basepoint), hence yield a map from the pushout
\begin{equation}\label{eqn:map-from-pushout}
    \frac{\rho^*Y}{\sim} \lra (N \sqcup E) \wedge_E p^*(S^\cL \wedge_X Y).
\end{equation}
The inclusion $E \hookrightarrow N \sqcup E$ and the given section $p^*\widetilde{u} \colon E \to p^*(S^\cL \wedge_X Y)$ also assemble into a map
\[
    E \lra (N \sqcup E) \wedge_E p^*(S^\cL \wedge_X Y).
\]
On the subspace $N \subset E$ it agrees with the composition of $\rho^*u$ and the map~\eqref{eqn:map-from-pushout}, hence we obtain a map
\begin{equation}\label{eqn:ansazt-comparison-map}
    E \cup_N \frac{\rho^*Y}{\sim} \lra (N \sqcup E) \wedge_E p^*(S^\cL \wedge_X Y).
\end{equation}

This is a map of spaces over $E$, so we may check that it is an equivalence after passing to homotopy fibres. Let $e \in E$, and write $h \colon F = \fib(N \to E) \to N$. Here and in the remainder of the proof, all homotopy fibres will be taken at $e$. From the following strongly cartesian cube
\[
\begin{tikzcd}[column sep={between origins, 5.5em}, row sep={between origins, 2em}]
 & \fib(\rho|_{\partial N}^*Y \to E) \arrow[lddd] \arrow[dd] \arrow[rr] & & \fib(\rho^*Y \to E) \arrow[dd] \\
 & & & \\
 & \rho|_{\partial N}^*Y \arrow[rr] \arrow[lddd] & & \rho^*Y \arrow[lddd]   \\
 \fib(\partial N \to E) \arrow[crossing over]{rr} \arrow[dd] & & F \arrow[dd]  & \\
 & & & \\
 \partial N \arrow[rr] & & N  &  
 \arrow[from=1-4, to=4-3, crossing over]
\end{tikzcd}
\]
we see that $\fib(\rho|_{\partial N}^*Y \to E) \simeq \fib(\partial N \to E) \times_F \fib(\rho^*Y \to E)$. Now using the pushout square defining $\frac{\rho^*Y}{\sim}$ and Mather's second cube theorem \cite[Theorem~25]{Mather1976}, we obtain a pushout square
\[
\begin{tikzcd}
\fib(\partial N \to E) \times_F \fib(\rho^*Y \to E) \arrow[d] \arrow[r] & \fib(\rho^*Y \to E) \arrow[d]    \\
\fib(\partial N \to E) \arrow[r]                                        & |[label={[label distance=-15mm]-45:\ulcorner}]|\fib(\frac{\rho^*Y}{\sim} \to E)
\end{tikzcd}
\]
Hence $\fib(\frac{\rho^*Y}{\sim} \to E) \simeq \fib(\partial N \to E) \ast_F \fib(\rho^*Y \to E)$ by definition of the fibrewise join $\ast_F$.
Likewise, from the pushout defining $E \cup_N \frac{\rho^*Y}{\sim}$ and Mather's second cube theorem, we get a pushout
\begin{equation}\label{eqn:mather-pushout}
\begin{tikzcd}
F \arrow[d] \arrow[r] & \fib(\partial N \to E) \ast_F \fib(\rho^*Y \to E) \arrow[d] \\
* \arrow[r]           & |[label={[label distance=-15mm]-45:\ulcorner}]|\fib(E \cup_N \frac{\rho^*Y}{\sim} \to E)                  
\end{tikzcd}    
\end{equation}
The inclusion $\partial N \hookrightarrow N$ is homotopy equivalent to the sphere bundle $S(\rho^*\cL)$ of the normal bundle. Hence, by pulling back, $\fib(\partial N \to E) \to F$ is equivalent to the sphere bundle $S(h^*\rho^*\cL) \to F$. From the pullback squares
\[
\begin{tikzcd}
|[label={[label distance=-2mm]-45:\lrcorner}]|\fib(\rho^*Y \to E) \arrow[r] \arrow[d] & |[label={[label distance=-2mm]-45:\lrcorner}]|F \arrow[r] \arrow[d, "h"] & * \arrow[d] \\
\rho^*Y \arrow[r]                       & N \arrow[r]           & E          
\end{tikzcd}
\]
we also see that $\fib(\rho^*Y \to E) \to F$ is equivalent to $h^*\rho^*Y \to F$. But this fibration is trivial because the composite
\[
    \rho \circ h \colon F \lra N \lra E \lra X
\]
is null-homotopic as $F$ is by definition the homotopy fibre of $N \to E$. Hence $h^*\rho^*Y \simeq F \times Y_x$ and $h^*\rho^*\cL \simeq F \times \cL_x$, where $x = p(e)$. Plugging this back into the pushout square~\eqref{eqn:mather-pushout}, we obtain
\begin{align*}
    \fib(E \cup_N \frac{\rho^*Y}{\sim} \to E) &\simeq \cofib(F \to (F \times S(\cL_x)) \ast_F (F \times Y_x)) \\
        &\simeq \cofib(F \to F \times (S(\cL_x) \ast Y_x)) \\
        &\simeq \cofib(F \to F \times (S^{\cL_x} \wedge Y_x)) \\
        &\simeq F_+ \wedge (S^{\cL_x} \wedge Y_x)
\end{align*}
which is exactly $\fib((N \sqcup E) \wedge_E p^*(S^{\cL} \wedge_X Y) \to E)$, and one can check that the map~\eqref{eqn:ansazt-comparison-map} restricts to such a homotopy equivalence on a fibre.
\end{proof}

Our first formula using two pushout squares was convenient for proving the lemma above, but we could have equivalently obtained $S(J^1\cL) \cup_N \frac{\rho^*Y}{\sim}$ from the product $S(J^1\cL) \times_X Y$ by crushing the second factor over the complement of the open tubular neighbourhood, i.e.
\begin{align*}
    S(J^1\cL) \cup_N \frac{\rho^*Y}{\sim} \cong \mathrm{pushout}\left(S(J^1\cL) \longleftarrow \mathring{N}^c \times_X Y \lra S(J^1\cL) \times_X Y\right).
\end{align*}
We shall simply denote this quotient by $\frac{S(J^1\cL) \times_X Y}{\sim}$.
Recall that the join of two topological spaces $A$ and $B$ is the space
\[
    A * B = \frac{A \times [0,1] \times B}{(a,0,b) \sim (a,0,b') \text{ and } (a,1,b) \sim (a',1,b)}.
\]
The fibrewise join is defined by the obvious generalisation of this formula. We shall write $[a,t,b]$ for a general element of the join, though $a \in A$ (resp. $b \in B$) is only well-defined if $t \neq 1$ (resp. $t \neq 0$). Because $J^1\cL \cong \cL \oplus \Omega^1_X \otimes \cL$, we have a homeomorphism $S(J^1\cL) \cong S(\cL) *_X S(\Omega^1_X \otimes \cL)$. This lets us define a natural collapse map:
\begin{lemma}\label{lemma:collapse-to-join}
The map
\begin{align*}
    S(\cL) *_X (S(\Omega^1_X \otimes \cL) \times_X Y) &\lra \frac{S(J^1\cL) \times_X Y}{\sim} \\
    [v,t,(w,y)] &\longmapsto [[v,t,w],y]
\end{align*}
is a homotopy equivalence of spaces over $X$.
\end{lemma}
\begin{proof}
One readily checks that the formula gives a well-defined map. Under any identification of $S(J^1\cL)$ with the join $S(\cL) *_X S(\Omega^1_X \otimes \cL)$, the closed neighbourhood $N$ corresponds to a subspace of the form $\set{[v,t,w]}{t \in [1-\epsilon, 1]}$ for some $0 < \epsilon < 1$. Then a homotopy inverse to the collapse map is given by
\begin{align*}
    \frac{S(J^1\cL) \times_X Y}{\sim} &\lra S(\cL) *_X (S(\Omega^1_X \otimes \cL) \times_X Y)  \\
     [[v,t,w],y] &\longmapsto [v,\max(0, t/\epsilon - 1/\epsilon  + 1),(w,y)],
\end{align*}
as one verifies using straight-line homotopies.
\end{proof}

Combining \cref{lemma:ansazt-nice-formula,lemma:collapse-to-join}, we have thus obtained:
\begin{proposition}\label{proposition:ansatz-is-join}
There is a weak homotopy equivalence of fibrations over $X$
\[
    \pushQED{\qed} 
    (N \sqcup S(J^1\cL)) \wedge_{S(J^1\cL)} p^*(S^\cL \wedge_X Y) \simeq S(\cL) *_X (S(\Omega^1_X \otimes \cL) \times_X Y). \qedhere
    \popQED
\]
\end{proposition}

\section{Proof of concept: Riemann surfaces in a background space}\label{section:POC}

The purpose of this section is to showcase our general method in the case of the moduli space of closed Riemann surfaces. Although the result is not new and was obtained by Cohen and Madsen in \cite{CohenMadsen}, and modern proofs can even deal with more general tangential structures \cite{OscarResolutionmoduli}, we find it instructive to give a new proof.

\bigskip

Let $S_g$ be a closed smooth surface of genus $g$. Denote by $\Diff^+(S_g)$ the topological group of orientation-preserving diffeomorphisms of $S_g$ with the $\cC^\infty$ topology. Let $Y$ be a space. As in \cite[Definition~1.1]{OscarResolutionmoduli}, we define the moduli space of closed Riemann surfaces in the background space $Y$ as the homotopy quotient:
\[
    \cM_g(Y) = \Map(S_g,Y) \sslash \Diff^+(S_g).
\]
Let $\MTSO(2)$ be the Madsen--Tillmann spectrum
\[
    \MTSO(2) = \mathrm{Th}(-\gamma_2 \to BSO(2)),
\]
or in other words the Thom spectrum of the virtual inverse of the tautological bundle $\gamma_2$ on $BSO(2)$. The parameterised Pontrjagin--Thom construction gives a natural transformation
\[
    \bZ \times \cM_g(Y) \lra \Omega^\infty(\MTSO(2) \otimes \Sigma^\infty_+Y)
\]
of functors of $Y$. For a construction of this map see e.g. \cite[Proof of Theorem~3.5]{CohenMadsen}, or combine \cite[Main theorem (with tangential structures)]{GMTW} and the inclusion of $\cM_g(Y)$ as a component of the space of maps of the cobordism category. (For the functorially minded reader, it can also be obtained quite formally using parameterised spectra as in \cite[Construction~5.19]{Azelie}.) Our goal is to prove the following theorem, using as inputs the Madsen--Weiss theorem \cite{MadsenWeiss} and Looijenga's computations of the stable homology of the moduli space of marked Riemann surfaces \cite{Looijenga}.
\begin{theorem}[{\cite[Theorem~0.1]{CohenMadsen}}]\label{theorem:surfaces-in-background}
Let $Y$ be a $2$-connected space. Then the map
\[
    \bZ \times \cM_g(Y) \lra \Omega^\infty(\MTSO(2) \otimes \Sigma^\infty_+Y)
\]
induces an isomorphism in integral homology in degrees $* \ll g$.
\end{theorem}

As explained in the introduction, and exemplified in this article, we will study the Goodwillie towers of the functors in the theorem. For reasons of convergence of the towers, we need to assume that $Y$ is $2$-connected, whereas \cite[Theorem~0.1]{CohenMadsen} only require $1$-connectedness.

\begin{lemma}\label{lemma:derivatives-exponential}
The derivatives of the functor $F = \Sigma^\infty \Omega^\infty$ based at $E \in \catSp$ are given by the spectra
\[
    \partial_i^E F \simeq \bS^0 \oplus F(E)
\]
with trivial $\Sigma_i$ action for $i \geq 1$.
\end{lemma}
\begin{proof}
The higher derivatives of $F$ agree with the ones of its associated reduced functor $\cofib(F(E) \to F)$. Combined with the equivalence $\catSp \cong \catSp_{E//E}$, $X \mapsto X \oplus E$, we see that it suffices to compute the derivatives of
\[
    \catSp \overset{\cong}{\lra} \catSp_{E//E} \overset{F}{\lra} \catSp_{F(E)//F(E)} \overset{\cofib}{\lra} \catSp
\]
sending $X \mapsto \cofib(F(E) \to F(X \oplus E))$. The standard splitting of the suspension of a product yields
\[
    F(X \oplus E) \simeq \Sigma^\infty \Omega^\infty (X \oplus E) \simeq \Sigma^\infty( \Omega^\infty X \times \Omega^\infty E) \simeq F(X) \oplus F(E) \oplus F(X) \otimes F(E),
\]
hence our composite functor is equivalent to $X \mapsto F(X) \otimes (\bS^0 \oplus F(E))$. It follows directly from the classification of homogeneous functors that, for any spectrum $Y$ and functor $G$, the derivative $\partial_i(Y \otimes G(-))$ is given by $Y \otimes \partial_i G$. Therefore
\[
    \partial_i^E F \simeq (\bS^0 \oplus F(E)) \otimes \partial_i^0 F
\]
where $\partial_i^0 F$ is the $i$th derivative of $F$ based at the zero spectrum. The lemma follows by directly computing using cross-effects that, for $i \geq 1$, $\partial_i^0 F \simeq \bS^0$ with the trivial $\Sigma_i$ action.
\end{proof}

We will now compute the derivatives of the target of the scanning map. Without loss of generality, we shall consider functors from the category of pointed spaces as it simplifies the theory of derivatives in the calculus.
\begin{corollary}\label{corollary:layers-mtso}
For $i \geq 1$, the $i$th derivative of the functor
\begin{align*}
    \catS_* &\lra  \catSp \\
    Y &\longmapsto \Sigma^\infty \Omega^\infty(\MTSO(2) \otimes \Sigma^\infty_+Y)
\end{align*}
is given by the spectrum
\[
    \partial_i = (\bS^0 \oplus \Sigma^\infty\Omega^\infty\MTSO(2)) \otimes \MTSO(2)^{\otimes i}
\]
together with the permutation action of the symmetric group $\Sigma_i$ on $\MTSO(2)^{\otimes i}$, and the trivial action on the other factor.
\end{corollary}
\begin{proof}
This functor is the composition of the reduced functor
\[
    \MTSO(2) \otimes \Sigma^\infty_+(-) \colon \catS_* \lra \catSp_{\MTSO(2)//\MTSO(2)}
\]
and the functor
\[
    \Sigma^\infty\Omega^\infty \colon \catSp_{\MTSO(2)//\MTSO(2)} \lra \catSp.
\]
By \cref{lemma:derivatives-exponential}, the $i$th derivative of the second functor is the spectrum $\bS^0 \oplus \Sigma^\infty\Omega^\infty\MTSO(2)$ with trivial $\Sigma_i$ action. Observe now that the first functor is linear with (first) derivative $\MTSO(2)$. We can thus conclude by using a very degenerate version of a chain rule: given a spectrum $E \in \catSp$ and $G \colon \catSp_{E//E} \to \catSp$ a stably $i$-excisive functor, the $i$th derivative of the composite $G \circ (E \otimes \Sigma^\infty_+(-))$ is $\partial_iG \otimes E^{\otimes i}$. The claim is clear if $G$ is homogeneous and follows by induction in general. (We have used a similar argument in the proof of \cref{lemma:pulled-back-tower}.)
\end{proof}

\begin{lemma}\label{lemma:layers-Mg}
For $i \geq 1$, the $i$th derivative of the functor 
\begin{align*}
    \catS_* &\lra \catSp \\
    Y &\longmapsto \Sigma^\infty_+\cM_g(Y)
\end{align*}
is given by the spectrum
\[
    \partial_i = \Conf_i(S_g)^{-iTS_g} \sslash \Diff^+(S_g)
\]
together with the natural permutation action of the symmetric group $\Sigma_i$ on the configurations and the summands of $iTS_g = (TS_g)^{\boxplus i}$.
\end{lemma}
\begin{proof}
For spectrum valued functors, taking derivatives commutes with homotopy colimits \cite[Proposition~1.18]{Goodwillie3}. Moreover, as a left adjoint $\Sigma^\infty_+(-)$ also commutes with homotopy colimits. Hence it suffices to compute the derivatives of the functor
\[
    Y \longmapsto \Sigma^\infty_+\Map(S_g,Y).
\]
But these are precisely given by Arone in \cite[Theorem~2]{Greg}, or by specialising \cref{theorem:goodwillie-approximation}.
\end{proof}

We are now ready to prove \cref{theorem:surfaces-in-background}. Let us stress one more time that we will be using as input the Madsen--Weiss theorem. Our contribution is merely to explain how, combined with Looijenga's computations and a bit of Goodwillie calculus, we can obtain a statement about the moduli of maps.
\begin{proof}[Proof of \cref{theorem:surfaces-in-background}]
The map of the theorem induces a morphism of Goodwillie towers:
\[
    P_\bullet(\Sigma^\infty_+(\bZ \times \cM_g(Y))) \lra P_\bullet(\Sigma^\infty_+\Omega^\infty(\MTSO(2) \otimes \Sigma^\infty_+Y))
\]
When $\conn(Y) \geq \dim(S_g) = 2$, the left tower converges because $\Sigma^\infty_+\Map(S_g,-)$ is $2$-analytic by the main result of \cite[Theorem~2]{Greg}. Directly from the definition of analyticity \cite[Definition~4.2]{Goodwillie2} and the Blakers--Massey theorem \cite[Theorem~2.4]{Goodwillie2}, we see that the tower of $\Sigma^\infty\Omega^\infty$ converges on connected spectra. Because the right tower is based at the $(-2)$-connective spectrum $\MTSO(2)$, it converges for $2$-connected spaces $Y$.
To prove our theorem, it then suffices to show that each map
\[
    P_n(\Sigma^\infty_+(\bZ \times \cM_g(Y))) \lra P_n(\Sigma^\infty_+\Omega^\infty(\MTSO(2) \otimes \Sigma^\infty_+Y))
\]
is an equivalence in degrees $* \ll g$. For $n=0$, this is the map
\[
    \Sigma^\infty_+(\bZ \times \cM_g) \lra \Sigma^\infty_+\Omega^\infty\MTSO(2)
\]
and the claim follows from the Madsen--Weiss theorem \cite[Theorem~1.1]{MadsenWeiss}. By induction and the 5-lemma, it remains to show that the induced map on the layers
\[
    D_n(\Sigma^\infty_+(\bZ \times \cM_g(Y))) \lra D_n(\Sigma^\infty_+\Omega^\infty(\MTSO(2) \otimes \Sigma^\infty_+Y))
\]
is highly connected for $n \geq 1$. By \cref{lemma:layers-Mg,corollary:layers-mtso}, it is itself induced by a map between the derivatives:
\[
    \Sigma^\infty_+\bZ \otimes \Conf_n(S_g)^{-nTS_g} \sslash \Diff^+(S_g) \lra (\bS^0 \oplus \Sigma^\infty\Omega^\infty\MTSO(2)) \otimes \MTSO(2)^{\otimes n}.
\]
As we are ultimately interested in a statement about integral homology, we may as well tensor both sides by the Eilenberg--MacLane spectrum $\bZ$. On the one hand, we obtain
\[
    \Sigma^\infty_+ \bZ \otimes \Conf_n(S_g)^{-nTS_g} \sslash \Diff^+(S_g) \otimes \bZ \simeq \Sigma^{-2n} \Sigma^\infty_+(\bZ \times \cM_{g,n}) \otimes \bZ
\]
where $\cM_{g,n} = \Conf_n(S_g) \sslash \Diff^+(S_g)$ is the moduli space of closed Riemann surfaces of genus $g$ with $n$ distinct marked points, and we have used the Thom isomorphism to untwist the Thom spectrum. (Carefully notice that in that formula $\Sigma^\infty_+\bZ$ is the suspension spectrum of the discrete space $\bZ$ whereas $\bZ$ is the Eilenberg--MacLane spectrum.) On the other hand, we have
\begin{align*}
    (\bS^0 \oplus \Sigma^\infty\Omega^\infty\MTSO(2) ) \otimes \MTSO(2)^{\otimes n} \otimes \bZ &\simeq \Sigma^\infty_+\Omega^\infty \MTSO(2) \otimes \MTSO(2)^{\otimes n} \otimes \bZ \\
    &\simeq \Sigma^{-2n} \Sigma^\infty_+(\Omega^\infty \MTSO(2) \times BSO(2)^n)  \otimes \bZ
\end{align*}
from splitting off the basepoint first, and then using the Thom isomorphism
\[
    \MTSO(2) \otimes \bZ \simeq \Sigma^{-2} \bZ \otimes \Sigma^\infty_+BSO(2).
\]
Looijenga has computed in \cite[Proposition~2.1]{Looijenga} that
\[
    H_*(\bZ \times \cM_{g,n}; \bZ) \cong H_*(\Omega^\infty\MTSO(2);\bZ)[u_1,\ldots,u_n] \cong H_*(\Omega^\infty\MTSO(2) \times BSO(2)^n;\bZ)
\]
in the stable range $* \ll g$. Looijenga states his theorem (proved prior to Madsen and Weiss' theorem) in terms of the stable cohomology of the mapping class group, which we now know is given by $H_*(\Omega^\infty\MTSO(2);\bZ)$ by \cite[Theorem~1.1]{MadsenWeiss}. Therefore we observe that the derivatives have isomorphic homology groups in the stable range. We leave it to the reader to chase through the computations to verify that the isomorphism is induced by the map on differentials. (Though, as \cref{section:comparison} suggests, this verification is not trivial.)
\end{proof}

\printbibliography

\end{document}